\documentclass[a4paper, 11pt]{amsproc}

\usepackage{fullpage}
\usepackage{amsmath, amsthm, amssymb, mathtools, mathrsfs}
\usepackage{ascmac}
\usepackage{comment}
\usepackage{bm}

\usepackage{hyperref}

\usepackage{graphicx}
\usepackage{here}
\usepackage{time}
\usepackage[abbrev]{amsrefs}

\usepackage{xcolor}
\usepackage[capitalize,nameinlink,noabbrev,nosort]{cleveref}
\hypersetup{
	colorlinks=true,       
	linkcolor=blue,          
	citecolor=blue,        
	filecolor=blue,      
	urlcolor=blue,           
}

\makeatletter
\@namedef{subjclassname@2020}{%
  \textup{2020} Mathematics Subject Classification}
\makeatother

\newtheorem{theoremcounter}{Theorem Counter}[section]

\theoremstyle{definition}
\newtheorem{definition}[theoremcounter]{Definition}

\newtheorem{example}[theoremcounter]{Example}

\theoremstyle{plain}
\newtheorem{lemma}[theoremcounter]{Lemma}
\newtheorem{proposition}[theoremcounter]{Proposition}
\newtheorem{corollary}[theoremcounter]{Corollary}

\newtheorem{theorem}[theoremcounter]{Theorem}

\numberwithin{equation}{section}

\newcommand{\Z}{\mathbb{Z}}
\newcommand{\Q}{\mathbb{Q}}
\newcommand{\R}{\mathbb{R}}
\newcommand{\C}{\mathbb{C}}

\newcommand{\bbH}{\mathbb{H}}

\DeclareMathOperator{\ImNew}{Im}
\renewcommand{\Im}{\ImNew}
\DeclareMathOperator{\ReNew}{Re}
\renewcommand{\Re}{\ReNew}

\DeclareMathOperator{\sgn}{sgn}

\DeclareMathOperator{\vol}{vol}

\newcommand{\pmat}[1]{\begin{pmatrix}#1\end{pmatrix}}

\newcommand{\smat}[1]{\bigl(\begin{smallmatrix}#1\end{smallmatrix}\bigr)}

\def\st#1#2{\genfrac{[}{]}{0pt}{}{#1}{#2}}

\begin{document}

\title[]{Hikami's observations on unified WRT invariants\\
and false theta functions} 

\author{Toshiki Matsusaka}
\address{Faculty of Mathematics, Kyushu University,
Motooka 744, Nishi-ku, Fukuoka 819-0395, Japan}
\email{matsusaka@math.kyushu-u.ac.jp} 




\maketitle

\emph{Dedicated to the memory of Toshie Takata.}

\begin{abstract}
	The object of this article is a family of $q$-series originating from Habiro's work on the Witten--Reshetikhin--Turaev invariants. The $q$-series usually make sense only when $q$ is a root of unity, but for some instances, it also determines a holomorphic function on the open unit disc. Such an example is Habiro's unified WRT invariant $H(q)$ for the Poincar\'{e} homology sphere. In 2007, Hikami observed its discontinuity at roots of unity. More precisely, the value of $H(\zeta)$ at a root of unity is $1/2$ times the limit value of $H(q)$ as $q$ tends towards $\zeta$ radially within the unit disc. In this article, we explain the appearance of the $1/2$-factor and generalize Hikami's observations by using Bailey's lemma and the theory of false theta functions.
\end{abstract}

\setcounter{tocdepth}{2}
\tableofcontents

\section{Introduction}

The WRT invariants are derived from the work of Witten~\cite{Witten1989} and Reshetikhin--Turaev~\cite{ReshetikhinTuraev1991}. Witten answered Atiyah's question on a $3$-dimensional definition of the Jones polynomials of knot theory and introduced certain invariants of $3$-manifolds using quantum field theory. Its rigorous mathematical definition was subsequently given by Reshetikhin and Turaev using the quantum group $U_q(sl_2)$ at roots of unity and has been extensively investigated.

Here is one example. The WRT-invariant $\tau_N(\Sigma(2,3,5))$ associated to the Poincar\'{e} homology sphere $M = \Sigma(2,3,5)$ is computed as
\begin{align*}
	e^{\frac{2\pi i}{N}\frac{121}{120}} (e^{\frac{2\pi i}{N}} - 1) \tau_N(\Sigma(2,3,5)) = \frac{e^{\pi i/4}}{2\sqrt{60N}} \sum_{\substack{0 \leq n \leq 60N-1 \\ N \nmid n}} e^{-\frac{\pi i n^2}{60N}} \frac{\prod_{j=1}^3 (e^{\frac{\pi in}{Np_j}} - e^{-\frac{\pi in}{Np_j}})}{e^{\frac{\pi in}{N}} - e^{-\frac{\pi in}{N}}}
\end{align*}
for $N \in \Z_{>1}$, where $(p_1, p_2, p_3) = (2,3,5)$ (see Lawrence--Rozansky~\cite{LawrenceRozansky1999} and Hikami~\cite{Hikami2005IJM}). One of the topics of research on the WRT invariants is to find a ``unified" function that can capture the values for all $N$. More precisely, we find a function $\tau(M): \mathcal{Z} \to \C$ defined on the set of all roots of unity $\mathcal{Z}$ such that the value $\tau(M)(e^{2\pi i/N})$ coincides with the WRT invariant $\tau_N(M)$. A number-theoretic (or analytic) approach was given by Lawrence--Zagier~\cite{LawrenceZagier1999} using false theta functions. They considered the $q$-series defined by
\[
	\widetilde{\Phi}_{(2,3,5)}^{(1,1,1)} (\tau) = \frac{1}{2} \sum_{n \in \Z} \sgn(n) \chi_{(2,3,5)}^{(1,1,1)}(n) q^{\frac{n^2}{120}} \qquad (q = e^{2\pi i\tau}),
\]
where $q^r = e^{2\pi i r \tau}$ for $r \in \Q$ and $\tau \in \bbH = \{\tau \in \C \mid \Im(\tau) > 0\}$, and 
\[
	\chi_{(2,3,5)}^{(1,1,1)} (n) = \begin{cases}
		1 &\text{if } n \equiv 31, 41, 49, 59 \pmod{60},\\
		-1 &\text{if } n \equiv 1, 11, 19, 29 \pmod{60},\\
		0 &\text{if otherwise}.
	\end{cases}
\]
Then they showed that
\[
	\lim_{t \to 0} \widetilde{\Phi}_{(2,3,5)}^{(1,1,1)} \left(\frac{1}{N} + it \right) = -\frac{1}{60N} \sum_{n=1}^{60N} n \chi_{(2,3,5)}^{(1,1,1)}(n) e^{\pi i \frac{n^2}{60N}}
\]
and
\[
	-\frac{e^{-\frac{\pi i}{60N}}}{2} \lim_{t \to 0} \widetilde{\Phi}_{(2,3,5)}^{(1,1,1)} \left(\frac{1}{N} + it \right) = 1 + e^{\frac{2\pi i}{N}} (1-e^{\frac{2\pi i}{N}}) \tau_N(\Sigma(2,3,5)).
\]
In this sense, the $q$-series $\widetilde{\Phi}_{(2,3,5)}^{(1,1,1)}(\tau)$ unifies the WRT-invariants via the limits to the roots of unity.

Another approach is developed by Habiro~\cite{Habiro2008}. Habiro constructed the unified WRT invariant $I_M(q)$ for the integral homology spheres $M$ with values in the set so-called ``Habiro ring" today. For instance, the unified WRT invariant $I_{\Sigma(2,3,5)}(q)$ he constructed is given by
\begin{align}\label{Habiro-235}
	H(q) := 1 + q(1-q) I_{\Sigma(2,3,5)}(q) = \sum_{n=0}^\infty q^n (q^n)_n,
\end{align}
where $(x)_n = (x; q)_n = \prod_{k=0}^{n-1} (1-xq^k)$ is the usual $q$-Pochhammer symbol. A characteristic of this type of series expression is that although an infinite sum defines it, substituting roots of unity for $q$ truncates the sum to a finite sum. Series with such properties were observed before Habiro. A few famous examples are Kontsevich's function $F(q) = \sum_{n=0}^\infty (q)_n$ studied in~\cite{Zagier2001} and Ramanujan's function $\sigma(q) = 1 + \sum_{n=1}^\infty (-1)^{n-1} q^n (q)_{n-1}$ discovered by Andrews~\cite{Andrews1986} and studied in~\cite{AndrewsDysonHickerson1988, Cohen1988}. In this case, $I_{\Sigma(2,3,5)}(e^{2\pi i/N}) = \tau_N(\Sigma(2,3,5))$, that is,
\[
	H(\zeta) = \sum_{n=0}^\infty \zeta^n (\zeta^n)_n = 1 + \zeta(1-\zeta) \tau_N(\Sigma(2,3,5))
\]
holds for $\zeta = e^{2\pi i/N}$.

Now we have two ways to unify the WRT invariants. Is there any direct relationship between them? First, it is worth noting that, by the term $q^n$ in the sum, Habiro's series in \eqref{Habiro-235} can be viewed as an element in $\Z[[q]]$, which is a feature not found in Kontsevich's function $F(q)$. Then, Hikami~\cite{Hikami2007} addressed this question and succeeded in showing the direct equation
\[
	H(q) = - q^{-\frac{1}{120}} \widetilde{\Phi}_{(2,3,5)}^{(1,1,1)}(\tau)
\]
as a holomorphic function on $|q| < 1$. However, we notice a strange phenomenon. By the above results, we see that
\begin{align}\label{Strange}
	H(e^{2\pi i/N}) = - \frac{1}{2} \lim_{\substack{\tau = 1/N + it \\ t \to 0}} q^{-\frac{1}{120}} \widetilde{\Phi}_{(2,3,5)}^{(1,1,1)}(\tau) = \frac{1}{2} \lim_{\substack{q = e^{2\pi i/N} e^{-t} \\ t \to 0}} H(q).
\end{align}
The mystery of the $1/2$-factor was pointed out by Habiro~\cite[Section 16]{Habiro2008}.

\subsection{Main results}

The article aims to generalize the relation between Habiro-type series and false theta functions studied by Hikami~\cite{Hikami2007} and provide a plausible explanation for the appearance of the $1/2$-factor. First, we review Hikami's results and observations.

For more general Brieskorn homology spheres $\Sigma(2,3,6p-1)$, Hikami explicitly expressed Habiro's unified WRT invariants as follows. For any integer $p > 1$, we have
\begin{align}\label{geometric-Habiro}
	H_p^{(1)}(q) := (1-q) I_{\Sigma(2,3,6p-1)}(q) = \sum_{s_p \geq \cdots \geq s_1 \geq 0} q^{s_p} (q^{s_p+1})_{s_p+1} \prod_{i=1}^{p-1} q^{s_i(s_i+1)} \st{s_{i+1}}{s_i}_q,
\end{align}
where $\st{\cdot}{\cdot}_q$ is the $q$-binomial coefficient defined by
\[
	\st{n}{k}_q = \frac{(q)_n}{(q)_k (q)_{n-k}}.
\]
Then substituting $q = e^{2\pi i/N}$ truncates the infinite sum defining the unified WRT invariant to a finite sum and $I_{\Sigma(2,3,6p-1)}(e^{2\pi i/N}) = \tau_N(\Sigma(2,3,6p-1))$ holds. On the other hand, Hikami~\cite{Hikami2005IJM} generalized Lawrence--Zagier's $q$-series as
\begin{align*}
	\widetilde{\Phi}_{(p_1, p_2, p_3)}^{(\ell_1,\ell_2,\ell_3)}(\tau) &= \frac{1}{2} \sum_{n \in \Z} \sgn(n)\chi_{(p_1,p_2,p_3)}^{(\ell_1,\ell_2, \ell_3)}(n) q^{\frac{n^2}{4p_1p_2p_3}}
\end{align*}
with a periodic function $\chi_{(p_1,p_2,p_3)}^{(\ell_1,\ell_2, \ell_3)}: \Z/2p_1p_2p_3\Z \to \{-1,0,1\}$, which we define later in \eqref{period-chi}. Then he showed that
\begin{align}\label{Hikami-Habiro-eq}
\begin{split}
	- \frac{1}{2} \lim_{\tau \to 1/N} q^{-\frac{(6p+5)^2}{24(6p-1)}} \widetilde{\Phi}_{(2,3,6p-1)}^{(1,1,1)}(\tau) &= (1-e^{2\pi i/N}) \tau_N(\Sigma(2,3,6p-1))\\
		&= H_p^{(1)}(e^{2\pi i/N})
\end{split}
\end{align}
for any $p > 1$. Here the limit is along the vertical line $\tau = 1/N + it$ as before. Similarly below, we will consider the vertical limit $\tau = 1/N + it \to 1/N$ or the radial limit $q = e^{2\pi i/N} e^{-t} \to e^{2\pi i/N}$ as limits. To observe a similarity to \eqref{Strange}, we are interested in comparing the limit 
\[
	\lim_{q \to e^{2\pi i/N}} H_p^{(1)}(q)
\]
from within the unit disc $|q| < 1$ and the value $H_p^{(1)}(e^{2\pi i/N})$ given in \eqref{Hikami-Habiro-eq}. In this case, however, numerical calculations show that the difference is no longer a constant multiple. More specifically, when $q$ tends to a root of unity from within the unit disc, we observe a divergence of $H_p^{(1)}(q)$. Our main theorem claims that the ``convergent part" of $H_p^{(1)}(q)$ converges to the value in \eqref{Hikami-Habiro-eq}.

\begin{theorem}[The precise statement is given in \cref{Habiro-false} and \cref{main}]\label{Intro-main}
	For any integer $p > 1$, as a holomorphic function on the open unit disc, the series has the expression
	\[
		H_p^{(1)}(q) = \frac{q^{-\frac{(6p+5)^2}{24(6p-1)}}}{2\eta(\tau)} \sum_{\bm{\varepsilon} = \smat{\varepsilon_1 \\ \varepsilon_2} \in \{0,1\}^2} (-1)^{\varepsilon_1 + \varepsilon_2} \bigg(\widetilde{\theta}_{\bm{\mu_1}+\bm{\varepsilon}, \bm{c_1}}^{(2)}(\tau) + \widetilde{\theta}_{\bm{\mu_1}+\bm{\varepsilon}, \bm{c_2}}^{(2)}(\tau) \bigg)
	\]
	in terms of false theta functions $\widetilde{\theta}_{\bm{\mu}+\bm{\varepsilon}, \bm{c}}^{(2)}(\tau)$ defined in \eqref{2-false-def} and the Dedekind eta function $\eta(\tau) = q^{1/24} (q)_\infty$. Then the first half converges in the vertical limit $\tau \to 1/N$ to
	\[
		\lim_{\tau \to 1/N} \frac{q^{-\frac{(6p+5)^2}{24(6p-1)}}}{2\eta(\tau)} \sum_{\bm{\varepsilon} \in \{0,1\}^2} (-1)^{\varepsilon_1 + \varepsilon_2} \widetilde{\theta}_{\bm{\mu_1}+\bm{\varepsilon}, \bm{c_1}}^{(2)}(\tau) = -\frac{1}{2} \lim_{\tau \to 1/N} q^{-\frac{(6p+5)^2}{24(6p-1)}} \widetilde{\Phi}_{(2,3,6p-1)}^{(1,1,1)} (\tau),
	\]
	which coincides with the value of $H_p^{(1)}(q)$ at $q = e^{2\pi i/N}$. The second half of the expression diverges in the same limit generally.
\end{theorem}

As for $p=1$, the function $H(q)$ given in \eqref{Habiro-235} is denoted by $H_1^{(2)}(q)$ in the following general notations. The series also has a similar expression
\[
	H_1^{(2)}(q) = \sum_{n=0}^\infty q^n(q^n)_n = \frac{q^{-\frac{1}{120}}}{2\eta(\tau)} \sum_{\bm{\varepsilon} \in \{0,1\}^2} (-1)^{\varepsilon_1 + \varepsilon_2} \bigg(\widetilde{\theta}_{\bm{\mu_2}+\bm{\varepsilon}, \bm{c_1}}^{(2)} (\tau) + \widetilde{\theta}_{\bm{\mu_2}+\bm{\varepsilon}, \bm{c_2}}^{(2)} (\tau) \bigg),
\]
and the first half converges to the value $H_1^{(2)}(e^{2\pi i/N})$ in the limit $\tau \to 1/N$. Furthermore, in this case, the first and second terms accidentally coincide. This fact follows from the symmetry of $a$ and $b$ in the expression given in \cref{Habiro-2}. Thus the limit of the whole $H_1^{(2)}(q)$ also converges, and its limit equals $2H_1^{(2)}(e^{2\pi i/N})$. That is a reason for the occurrence of the $1/2$-factor in \eqref{Strange}.

Hikami~\cite{Hikami2007} also gave many observations on the relations between other Habiro-type series and the limits of $\widetilde{\Phi}_{(2,3,6p-1)}^{(\ell_1, \ell_2, \ell_3)}(\tau)$. More precisely, he introduced another infinite family of Habiro-type series $H_p^{(5)}(q)$ and three more examples $H_2^{(2)}(q), H_2^{(3)}(q)$, and $H_2^{(4)}(q)$ in the following notations. Here we generalize Hikami's examples to five infinite families.

\begin{definition}\label{Habiro-elements}
	For any positive integer $p \geq 1$, we define five Habiro-type series by
	\allowdisplaybreaks[1]
	\begin{align*}
		H_p^{(1)}(q) &= \sum_{s_p \geq \cdots \geq s_1 \geq 0} q^{s_p} (q^{s_p+1})_{s_p+1} \prod_{i=1}^{p-1} q^{s_i(s_i+1)} \st{s_{i+1}}{s_i}_q,\\
		H_p^{(2)}(q) &= \sum_{s_p \geq \cdots \geq s_1 \geq 0} q^{s_p} (q^{s_p})_{s_p} \prod_{i=1}^{p-1} q^{s_i^2} \st{s_{i+1}}{s_i}_q,\\
		H_p^{(3)}(q) &= \sum_{s_p \geq \cdots \geq s_1 \geq 0} q^{2s_p} (q^{s_p+1})_{s_p} \prod_{i=1}^{p-1} q^{s_i(s_i+1)} \st{s_{i+1}}{s_i}_q,\\
		H_p^{(4)}(q) &= \sum_{s_p \geq \cdots \geq s_1 \geq 0} q^{s_p} (q^{s_p+1})_{s_p} \prod_{i=1}^{p-1} q^{s_i(s_i+1)} \st{s_{i+1}}{s_i}_q,\\
		H_p^{(5)}(q) &= \sum_{s_p \geq \cdots \geq s_1 \geq 0} q^{s_p} (q^{s_p+1})_{s_p} \prod_{i=1}^{p-1} q^{s_i^2} \st{s_{i+1}}{s_i}_q.
	\end{align*}
\end{definition}

If the notations are to match those adapted by the spirit of Hikami~\cite{Hikami2007}, then the above series should be named $H_p^{(1)}(q) = M_p^{(1)}(q)$, $H_p^{(2)}(q) = M_p^{(p)}(q)$, $H_p^{(3)}(q) = M_p^{(2p-1)}(q)$, $H_p^{(4)}(q) = M_p^{(2p)}(q)$, and $H_p^{(5)}(q) = M_p^{(3p-1)}(q)$. However, since the superscripts overlap when $p=1$, the notations here are purposely changed. These five series are infinite families that extend each of Hikami's $M_2^{(k)}(q)$ for $k=1,2,3,4,5$. Moreover, $H_1^{(2)}(q) = M_1^{(1)}(q)$ and $H_1^{(4)}(q) = H_1^{(5)}(q) = M_1^{(2)}(q)$ hold in Hikami's notations.

Our main theorems stated in \cref{Habiro-false} and \cref{main} give similar expressions in terms of false theta functions and limit formulas of these five families as in \cref{Intro-main}. For instance, we have
\begin{align*}
	\lim_{\tau \to 1/N} \text{convergent part of }H_2^{(2)}(q) &= -\frac{1}{2} \lim_{\tau \to 1/N} q^{-\frac{1}{264}} \widetilde{\Phi}_{(2,3,11)}^{(1,1,2)}(\tau)\\
		&= \frac{e^{-\frac{2\pi i}{264N}}}{264N} \sum_{n=1}^{132N} n \chi_{(2,3,11)}^{(1,1,2)}(n) e^{\pi i \frac{n^2}{132N}},
\end{align*}
where
\[
	\chi_{(2,3,11)}^{(1,1,2)} (n) = \begin{cases}
		1 &\text{if } n \equiv 67,89,109,131 \pmod{132},\\
		-1 &\text{if } n \equiv 1, 23,43,65 \pmod{132},\\
		0 &\text{if otherwise}.
	\end{cases}
\]
Moreover, numerical calculations suggest that the above limit value coincides with the value $H_2^{(2)}(e^{2\pi i/N})$, that is,
\begin{align}
	H_2^{(2)}(e^{2\pi i/N}) = \frac{e^{-\frac{2\pi i}{264N}}}{264N} \sum_{n=1}^{132N} n \chi_{(2,3,11)}^{(1,1,2)}(n) e^{\pi i \frac{n^2}{132N}}
\end{align}
holds. The similarity with \cref{Intro-main} leads us to expect the coincidence to hold, but it is a conjecture. For other cases, too, Hikami~\cite[Conjectures 1--3]{Hikami2007} conjectured the coincidence between the limits of $\widetilde{\Phi}_{(2,3,6p-1)}^{(\ell_1, \ell_2, \ell_3)}(\tau)$ and the values of Habiro-type series through numerical calculations, but they are still open problems.

To conclude this introduction section, we introduce some related studies. First, Hikami also studied the unified WRT invariants for the Brieskorn homology spheres $\Sigma(2,3,6p+1)$ with $p \geq 1$, but we do not deal with the cases in this article. Second, the above conjecture for $H_p^{(1)}(q)$,
\[
	\bigg(\lim_{\tau \to 1/N} \text{convergent part of } H_p^{(1)}(q) \bigg) = \bigg( \text{the value of } H_p^{(1)}(q) \text{ at } q = e^{2\pi i/N} \bigg)
\]
proved in \cref{Intro-main}, is derived from the fact that both sides have the topological interpretations \eqref{geometric-Habiro} and \eqref{Hikami-Habiro-eq}, namely begin the WRT invariants. On the other hand, the Habiro-type series defined in \cref{Habiro-elements} are found by numerical experiments so that the analogy of \cref{Intro-main} holds, and so far, its roles in the theory of WRT invariants are unclear. Third, many other known methods exist to unify the WRT invariants for more general $3$-manifolds. For instance, Hikami~\cite{Hikami2006JMP} further generalized Lawrence--Zagier's series $\widetilde{\Phi}_{\bm
{p}}^{\bm{\ell}}(\tau)$ for the Seifert fibered homology $3$-spheres with $n$-singular fibers. More recently, Gukov--Pei--Putrov--Vafa~\cite{GPPV2020} introduced $q$-series called homological blocks for any plumbed $3$-manifolds associated with negative definite plumbing tree graphs based on Gukov--Putrov--Vafa~\cite{GukovPutrovVafa2017}. Andersen--Misteg\aa rd~\cite{AM2022} and Fuji--Iwaki--H. Murakami--Terashima~\cite{FIMT2021} independently studied the limit of the homological blocks at roots of unity in different contexts and showed that the homological blocks also unify the WRT invariants for Seifert fibered integral homology $3$-spheres. As for other manifolds, Mori--Y. Murakami~\cite{MoriMurakami2022} dealt with the case for the $\mathrm{H}$-graph, and Y. Murakami~\cite{Murakami2022+} extended it to more general cases. Furthermore, the modular transformation theory for the homological blocks is developing by Bringmann--Mahlburg--Milas~\cite{BMM2020}, Bringmann--Kaszian--Milas--Nazaroglu~\cite{BKMN2023}, and Matsusaka--Terashima~\cite{MatsusakaTerashima2021} et al.

This article is organized as follows. In \cref{s2}, we give Hecke-type series expressions of the five families of the Habiro-type series. This expression yields the relation between the Habiro-type series and the false theta functions. Since the key to the proof is Bailey's work on the Rogers--Ramanujan identities, we begin by reviewing it in the first half of \cref{s2}. In \cref{s3}, we introduce the notion of the false theta functions based on the recent work of Bringmann--Nazaroglu~\cite{BringmannNazaroglu2019}. Then, under this setting, we review Hikami's work~\cite{Hikami2005IJM} on the function $\widetilde{\Phi}_{\bm{p}}^{\bm{\ell}}(\tau)$. In \cref{s3-2}, we give our first main theorem (\cref{Habiro-false}) on the expressions of the Habiro-type series in terms of the false theta functions $\widetilde{\theta}_{\bm{\mu}, \bm{c}}^{(2)}(\tau)$. The transformation called ``false theta decomposition" decomposes the false theta functions into a sum of products of the (lower-dimensional) false theta functions $\widetilde{\theta}_{M, \mu}^{(1)}(\tau)$ and the ordinary theta functions $\theta_{M,\mu}^{(1)}(\tau)$ (\cref{false-decompose1}). This decomposition allows us to calculate the limit of Habiro-type series at roots of unity and obtain our second main theorem (\cref{main}). Finally, we revisit Hikami's question on the modular transformation theory of the Hecke-type series related to the Habiro-type series.

\section{Hecke-type formulas}\label{s2}

This section aims to transform the five Habiro-type series defined in \cref{Habiro-elements} into a Hecke-type series. Since the basic idea is based on Bailey's lemma, developed by Andrews, we review it first. Then, as an application of Bailey's lemma, we show five critical identities related to the Habiro-type series in \cref{p-Bailey-result} and a series of lemmas. In \cref{s2-3}, we derive the desired Hecke-type expressions. Finally, in \cref{s2-4}, although off-topic, we remark on a well-known equation of multiple zeta values derived from Bailey's transform.

\subsection{Bailey's lemma}

Bailey's lemma has a long history, dating back to Bailey's work~\cite{Bailey1947, Bailey1948} in the 1940s, which clarifies the structure of Rogers' second proof of the Rogers--Ramanujan identities. The original idea of Bailey is simple but has several powerful applications. For example, Andrews~\cite{Andrews1986TAMS} found Hecke-type formulas of Ramanujan's mock theta functions by constructing particular Bailey pairs. This discovery by Andrews led Zwegers~\cite{Zwegers2002} to establish the modular transformation theory of mock theta functions. In this subsection, we recall the claims and ideas of Bailey's transform and Bailey's lemma. Its more detailed and extensive history can be found in Andrews~\cite{Andrews1986AMS}, Warnaar~\cite{Warnaar1999}, and Sills~\cite{Sills2018}.

\begin{lemma}[{Bailey's transform}]\label{Baileytransform}
	If sequences $(\alpha_n)_n, (\beta_n)_n, (\gamma_n)_n, (\delta_n)_n, (u_n)_n$, and $(v_n)_n$ satisfy suitable convergence conditions and the equations
	\begin{align*}
		\beta_n = \sum_{k=0}^n \alpha_k u_{n-k} v_{n+k}, \qquad \gamma_n = \sum_{k=n}^\infty \delta_k u_{k-n} v_{k+n},
	\end{align*}
	then we have
	\[
		\sum_{n=0}^\infty \alpha_n \gamma_n = \sum_{n=0}^\infty \beta_n \delta_n.
	\]
\end{lemma}

The proof is simply an exchange of the order of the sums, where the ``suitable convergence conditions" are required. In particular, let us choose $u_n = 1/(q)_n$ and $v_n = 1/(aq)_n$ with a complex number $a \in \C$. Here $(x)_n = (x; q)_n = \prod_{k=0}^{n-1} (1-xq^k)$ is the usual $q$-Pochhammer symbol with $|q| < 1$. Then the four sequences are required to satisfy the following equations.
\begin{align}\label{Bailey-transform}
	\beta_n = \sum_{k=0}^n \frac{\alpha_k}{(q)_{n-k} (aq)_{n+k}}, \qquad \gamma_n = \sum_{k=n}^\infty \frac{\delta_k}{(q)_{k-n} (aq)_{k+n}}.
\end{align}
A pair of sequences $(\alpha, \beta)$ satisfying the above first equation is called a \emph{Bailey pair relative to $a$}. Similarly, a pair $(\gamma, \delta)$ satisfying the second equation is called a \emph{conjugate Bailey pair relative to $a$}.

In applications, Bailey~\cite[\S.4]{Bailey1948} found the following conjugate Bailey pair $(\gamma, \delta)$.
\begin{lemma}\label{conjugate-Bailey-pair}
	For any $\rho_1, \rho_2 \in \C$ (such that no zeros appear in the denominators) and a non-negative integer $N \geq 0$, a pair of
	\begin{align*}
		\gamma_n &= \frac{(aq/\rho_1)_N (aq/\rho_2)_N}{(aq)_N (aq/\rho_1 \rho_2)_N} \frac{(-1)^n (\rho_1)_n (\rho_2)_n (q^{-N})_n}{(aq/\rho_1)_n (aq/\rho_2)_n (aq^{N+1})_n} \left( \frac{aq}{\rho_1 \rho_2}\right)^n q^{nN - \frac{n(n-1)}{2}},\\
		\delta_n &= \frac{(\rho_1)_n (\rho_2)_n (q^{-N})_n q^n}{(\rho_1 \rho_2 q^{-N}/a)_n}
	\end{align*}
	is a conjugate Bailey pair relative to $a$.
\end{lemma}

\begin{proof}
	The key to the proof is $q$-analogue of the Saalsch\"{u}tz summation formula for the $q$-hypergeometric series ${}_3 \phi_2$. The proof can be found in Andrews~\cite[p.25--27]{Andrews1986AMS}.
\end{proof}

Since $\gamma_n = \delta_n = 0$ for $n>N$, the ``suitable convergence conditions" required in \cref{Baileytransform} is satisfied. The following Bailey's lemma tells us that a Bailey pair $(\alpha, \beta)$ yields a new Bailey pair $(\alpha', \beta')$.

\begin{theorem}[{Bailey's lemma}]
	If $(\alpha, \beta)$ is a Bailey pair relative to $a$, then a pair of
	\begin{align*}
		\alpha'_n &= \frac{(\rho_1)_n (\rho_2)_n \left(\frac{aq}{\rho_1 \rho_2}\right)^n}{(aq/\rho_1)_n (aq/\rho_2)_n} \alpha_n,\\
		\beta'_n &= \sum_{j=0}^n \frac{(\rho_1)_j (\rho_2)_j (aq/\rho_1 \rho_2)_{n-j} \left(\frac{aq}{\rho_1 \rho_2}\right)^j}{(q)_{n-j} (aq/\rho_1)_n (aq/\rho_2)_n} \beta_j
	\end{align*}
	is also a Bailey pair relative to $a$, that is,
	\begin{align*}
		\beta'_n = \sum_{k=0}^n \frac{\alpha'_k}{(q)_{n-k} (aq)_{n+k}}
	\end{align*}
	holds.
\end{theorem}

\begin{proof}
	We give a sketch of the proof given in~\cite[p.27]{Andrews1986AMS}. A direct calculation yields
	\begin{align*}
		\sum_{k=0}^N \frac{\alpha'_k}{(q)_{N-k} (aq)_{N+k}} &= \frac{(aq/\rho_1 \rho_2)_N}{(aq/\rho_1)_N (aq/\rho_2)_N (q)_N} \sum_{k=0}^\infty \gamma_k \alpha_k,
	\end{align*}
	where $\gamma_k$ is defined in \cref{conjugate-Bailey-pair} and $\gamma_k=0$ for $k > N$. By \cref{Baileytransform}, the last sum equals $\sum_{k=0}^N \beta_k \delta_k$. A more direct calculation yields the definition of $\beta'_N$.
\end{proof}

For later applications, we will compute the particular case of Bailey's lemma.

\begin{corollary}\label{Bailey-lemma-special}
	If $(\alpha, \beta)$ is a Bailey pair relative to $a$, then a pair of
	\begin{align*}
		\alpha'_n = a^n q^{n^2} \alpha_n, \qquad \beta'_n = \sum_{j=0}^n \frac{a^j q^{j^2}}{(q)_{n-j}} \beta_j
	\end{align*}
	is also a Bailey pair relative to $a$.
\end{corollary}

\begin{proof}
	In the definition of $(\alpha'_n, \beta'_n)$, we take a limit as $\rho_1, \rho_2 \to \infty$.
\end{proof}

\begin{example}\label{Ex:RR-id}
	We explain how Rogers--Ramanujan's identities follow from Bailey's lemma. The most basic example of Bailey pairs (relative to $a$) is the \emph{unit Bailey pair} defined by
	\begin{align}\label{unite-Bailey}
	\begin{split}
		\alpha_n &= \frac{(1-aq^{2n}) (a)_n (-1)^n q^{\frac{n(n-1)}{2}}}{(1-a) (q)_n},\\
		\beta_n &= \delta_{n,0} = \begin{cases}
			1 &\text{if } n=0,\\
			0 &\text{if } n>0
		\end{cases}
	\end{split}
	\end{align}
	(see Andrews~\cite[(2.12) and (2.13)]{Andrews1984PJM}). First, we let $a=1$. By applying \cref{Bailey-lemma-special} twice, we see that a pair of
	\begin{align}
	\begin{split}
		\alpha''_n &= q^{2n^2} \alpha_n = \begin{cases}
			1 &\text{if } n =0,\\
			(-1)^n q^{\frac{n(5n-1)}{2}} (1+q^n) &\text{if } n>0,
		\end{cases}\\
		\beta''_n &= \sum_{j=0}^n \frac{q^{j^2}}{(q)_{n-j}} \sum_{k=0}^j \frac{q^{k^2}}{(q)_{j-k}} \beta_k = \sum_{j=0}^n \frac{q^{j^2}}{(q)_{n-j} (q)_j}
	\end{split}
	\end{align}
	is also a Bailey pair relative to $1$. By the first relation in \eqref{Bailey-transform} and taking a limit as $n \to \infty$, we have
	\begin{align*}
		\sum_{j=0}^\infty \frac{q^{j^2}}{(q)_j} = \frac{1}{(q)_\infty} \sum_{k \in \Z} (-1)^k q^{\frac{k(5k-1)}{2}}.
	\end{align*}
	By Jacobi's triple product
	\begin{align}\label{Jacobi-triple}
		\sum_{n \in \Z} (-1)^n q^{\frac{n(n-1)}{2}} \zeta^n = (q)_\infty (\zeta)_\infty (\zeta^{-1}q)_\infty,
	\end{align}
	we have
	\[
		\sum_{j=0}^\infty \frac{q^{j^2}}{(q)_j} = \frac{(q^5; q^5)_\infty (q^2; q^5)_\infty (q^3; q^5)_\infty}{(q;q)_\infty} = \frac{1}{(q;q^5)_\infty (q^4; q^5)_\infty},
	\]
	which is so-called \emph{Rogers--Ramanujan's first identity}.
	
	Similarly, we let $a=q$ in \eqref{unite-Bailey}. Again, from the twice application of \cref{Bailey-lemma-special} and Jacobi's triple product, we obtain the \emph{second Rogers--Ramanujan identity},
	\[
		\sum_{j=0}^\infty \frac{q^{j(j+1)}}{(q)_j} = \frac{1}{(q)_\infty} \sum_{k \in \Z} (-1)^k q^{\frac{k(5k + 3)}{2}} = \frac{1}{(q^2; q^5)_\infty (q^3; q^5)_\infty}.
	\]
\end{example}

\subsection{Bailey chains related to the Rogers--Ramanujan identities}\label{s2-2}

As seen in \cref{Ex:RR-id}, repeated application of Bailey's lemma yields a sequence of Bailey pairs. We call the sequence a \emph{Bailey chain}. To obtain Hecke-type expansions of the Habiro-type series defined in \cref{Habiro-elements}, we recall two Bailey chains considered by Hikami~\cite{Hikami2007} and show three auxiliary lemmas below.

Let $\st{\cdot}{\cdot}_q$ be the $q$-binomial coefficient defined by
\[
	\st{n}{k}_q = \frac{(q)_n}{(q)_k (q)_{n-k}}.
\]

The first Bailey chain follows from a unit Bailey pair with $a=1$ considered in \cref{Ex:RR-id} related to the first Rogers--Ramanujan identity.

\begin{proposition}\label{Bailey-chain-1}
	For any integer $p \geq 1$, a pair $(\alpha^{(1,p)}, \beta^{(1,p)})$ defined by
	\begin{align*}
		\alpha_n^{(1,p)} &= (-1)^n q^{\left(p+\frac{1}{2}\right)n^2 -\frac{1}{2}n} (1+q^n) - \delta_{n,0},\\
		\beta_n^{(1,p)} &= \frac{1}{(q)_n} \sum_{n = s_p \geq \cdots \geq s_1 \geq 0} \prod_{i=1}^{p-1} q^{s_i^2} \st{s_{i+1}}{s_i}_q
	\end{align*}
	is a Bailey pair relative to $1$, where the empty product is understood as $1$, that is, $\beta_n^{(1,1)} = 1/(q)_n$.
\end{proposition}

\begin{proof}
	It follows from induction on $p$. The initial case of $p=1$ is given by applying \cref{Bailey-lemma-special} to the unit Bailey pair with $a=1$. By the definition, we see that
	\begin{align*}
		\alpha_n^{(1,p)} = q^{n^2} \alpha_n^{(1,p-1)}, \qquad \beta_n^{(1,p)} = \sum_{s_{p-1}=0}^n \frac{q^{s_{p-1}^2}}{(q)_{n-s_{p-1}}} \beta_{s_{p-1}}^{(1,p-1)}.
	\end{align*}
	The induction assumption and Bailey's lemma in~\cref{Bailey-lemma-special} imply that $(\alpha^{(1,p)}, \beta^{(1,p)})$ is also a Bailey pair relative to $1$.
\end{proof}

The following second Bailey chain is related to the second Rogers--Ramanujan identity.

\begin{proposition}\label{Bailey-chain-2}
	For any integer $p \geq 1$, a pair $(\alpha^{(2,p)}, \beta^{(2,p)})$ defined by
	\begin{align*}
		\alpha_n^{(2,p)} &= (-1)^n q^{\left(p+\frac{1}{2}\right)n^2 + \left(p-\frac{1}{2}\right)n} \frac{1-q^{2n+1}}{1-q},\\
		\beta_n^{(2,p)} &= \frac{1}{(q)_n} \sum_{n = s_p \geq \cdots \geq s_1 \geq 0} \prod_{i=1}^{p-1} q^{s_i(s_i+1)} \st{s_{i+1}}{s_i}_q
	\end{align*}
	is a Bailey pair relative to $q$.
\end{proposition}

\begin{proof}
	The proof is the same as that of \cref{Bailey-chain-1}. More precisely, the claim follows from repeatedly applying \cref{Bailey-lemma-special} to the unit Bailey pair with $a=q$.
\end{proof}

Returning to the definition of Bailey pairs, the above two propositions are equivalent to the following equations.

\begin{proposition}\label{p-Bailey-result}
	For any integer $p \geq 1$, we have
	\begin{align*}
		\beta_n^{(1,p)} &= \sum_{k=-n}^n \frac{(-1)^k q^{\left(p+\frac{1}{2}\right)k^2 - \frac{1}{2}k}}{(q)_{n-k} (q)_{n+k}},\\
		\beta_n^{(2,p)} &= \sum_{k=-n-1}^n \frac{(-1)^k q^{\left(p+\frac{1}{2}\right)k^2 + \left(p-\frac{1}{2}\right)k}}{(q)_{n-k} (q)_{n+k+1}}.
	\end{align*}
\end{proposition}

\begin{proof}
	The first equation follows from a straightforward calculation. As for the second equation, we have
	\begin{align*}
		\beta_n^{(2,p)} = \sum_{k=0}^n \frac{\alpha_k^{(2,p)}}{(q)_{n-k} (q^2)_{n+k}} = \sum_{k=0}^n \frac{(-1)^k q^{\left(p+\frac{1}{2}\right)k^2 + \left(p-\frac{1}{2}\right)k}}{(q)_{n-k} (q)_{n+k+1}} (1-q^{2k+1}).
	\end{align*}
	We obtain the desired equation by dividing the sum into two parts and changing the variable.
\end{proof}


As can be immediately expected from the definitions of $\beta_n^{(1,p)}$ and $\beta_n^{(2,p)}$, these are closely related to the Habiro-type series. For later calculations, we prepare three auxiliary lemmas.

\begin{lemma}\label{auxiliary1}
	For any integer $n \geq 1$, we have
	\[
		(1-q^n) \beta_n^{(1,p)} = (1-q^n) \sum_{k=-n}^n \frac{(-1)^k q^{\left(p+\frac{1}{2}\right)k^2 - \frac{1}{2}k}}{(q)_{n-k} (q)_{n+k}} = \sum_{k=-n+1}^n \frac{(-1)^k q^{\left(p+\frac{1}{2}\right)k^2 - \frac{1}{2}k}}{(q)_{n-k} (q)_{n+k-1}}.
	\]
\end{lemma}

\begin{proof}
	We put
	\begin{align*}
		a_{n,k} &= (1-q^n) \frac{(-1)^k q^{\left(p+\frac{1}{2}\right)k^2 - \frac{1}{2}k}}{(q)_{n-k} (q)_{n+k}} \quad (-n \leq k \leq n),\\
		b_{n,k} &= \frac{(-1)^k q^{\left(p+\frac{1}{2}\right)k^2 - \frac{1}{2}k}}{(q)_{n-k} (q)_{n+k-1}} \quad (-n+1 \leq k \leq n).
	\end{align*}
	Then we can quickly check that $a_{n,-n} + a_{n,n} = b_{n,n}$ and $a_{n,-k} + a_{n,k} = b_{n,-k} + b_{n,k}$ for any $0 \leq k \leq n-1$, which concludes the proof.
\end{proof}

\begin{lemma}\label{auxiliary2}
	For any integer $n \geq 0$, we have
	\[
		q^n \beta_n^{(2,p)} = q^{n} \sum_{k=-n-1}^n \frac{(-1)^k q^{\left(p+\frac{1}{2}\right)k^2+\left(p-\frac{1}{2}\right)k}}{(q)_{n-k} (q)_{n+k+1}} = \sum_{k=-n}^n \frac{(-1)^k q^{\left(p+\frac{1}{2}\right)k^2+\left(p+\frac{1}{2}\right)k}}{(q)_{n-k} (q)_{n+k}}.
	\]
\end{lemma}

\begin{proof}
	We put
	\begin{align*}
		a_{n,k} &= q^n \frac{(-1)^k q^{\left(p+\frac{1}{2}\right)k^2+\left(p-\frac{1}{2}\right)k}}{(q)_{n-k} (q)_{n+k+1}} \quad (-n-1 \leq k \leq n),\\
		b_{n,k} &= \frac{(-1)^k q^{\left(p+\frac{1}{2}\right)k^2+\left(p+\frac{1}{2}\right)k}}{(q)_{n-k} (q)_{n+k}} \quad (-n \leq k \leq n).
	\end{align*}
	Again, we can also show by a direct calculation that $a_{n,-n-1} + a_{n,n} = b_{n,n}$ and $a_{n,-k-1} + a_{n,k} = b_{n,-k-1} + b_{n,k}$ for any $0 \leq k \leq n-1$.
\end{proof}

\begin{lemma}\label{auxiliary3}
	For any integer $n \geq 1$, we have
	\[
		(1-q^{2n}) \beta_n^{(2,p)} = (1-q^{2n}) \sum_{k=-n-1}^n \frac{(-1)^k q^{\left(p+\frac{1}{2}\right)k^2+\left(p-\frac{1}{2}\right)k}}{(q)_{n-k} (q)_{n+k+1}} = \sum_{k=-n+1}^n \frac{(-1)^k q^{\left(p+\frac{1}{2}\right)k^2+\left(p-\frac{1}{2}\right)k}}{(q)_{n-k} (q)_{n+k-1}}.
	\]
\end{lemma}

\begin{proof}
	As in the above two proofs, we put
	\begin{align*}
		a_{n,k} &= (1-q^{2n}) \frac{(-1)^k q^{\left(p+\frac{1}{2}\right)k^2+\left(p-\frac{1}{2}\right)k}}{(q)_{n-k} (q)_{n+k+1}} \quad (-n-1 \leq k \leq n),\\
		b_{n,k} &= \frac{(-1)^k q^{\left(p+\frac{1}{2}\right)k^2+\left(p-\frac{1}{2}\right)k}}{(q)_{n-k} (q)_{n+k-1}} \quad (-n+1 \leq k \leq n).
	\end{align*}
	The situation becomes slightly different, but $a_{n,-n-1} + a_{n,n} = b_{n,n}$, $a_{n,-n} + a_{n,n-1} = b_{n,n-1}$, and $a_{n,-k-1} + a_{n,k} = b_{n,-k-1} + b_{n,k}$ $(0 \leq k \leq n-2)$ are still shown in the same way.
\end{proof}

\subsection{Hecke-type expansions of Habiro-type series}\label{s2-3}

Using the Bailey chains prepared in the previous subsection, we transform the five Habiro-type series into the Hecke-type series. Since Hikami~\cite{Hikami2007} has done the first and the fifth cases, we will review his proofs and then work on the remaining three. We note that for the remaining three, Hikami led to Hecke-type expansions only for the case $p=2$. Our result extends Hikami's results, but our proof method is slightly different from his.

The following equation is also used in the proof (see Hikami~\cite[Lemma 3.6]{Hikami2007}).

\begin{lemma}\label{Fine}
	For any non-negative integer $c \geq 0$, we have
	\[
	\sum_{m=0}^\infty q^m \frac{(q)_{2m+c}}{(q)_{m+c} (q)_m} = \frac{1}{(q)_\infty} \sum_{k=0}^\infty (-1)^k q^{\frac{3}{2} k^2 + \left(c+\frac{3}{2} \right)k}.
\]
\end{lemma}

\begin{proof}
	In the equation given in \cite[(25.96)]{Fine1988},
	\[
		\sum_{m=0}^\infty \frac{(aq)_{2m} (bq)_m}{(aq)_m (q)_m} t^m = \frac{(btq)_\infty}{(t)_\infty} \sum_{k=0}^\infty \frac{(bq)_k (t)_k}{(q)_k (btq)_{2k}} (-at)^k q^{\frac{k(3k+1)}{2}},
	\]
	we take $a = q^c, b=0$ and $t = q$.
\end{proof}

\begin{theorem}[{Hikami~\cite[Theorem 3.5]{Hikami2007}}]\label{Hikami-Theorem3.5}
	For $p \geq 1$, we have the Hecke-type expansion for $H_p^{(1)}(q)$ as
	\[
		H_p^{(1)}(q) = \frac{1}{(q)_\infty} \left(\sum_{a, b \geq 0} - \sum_{a, b < 0} \right) (-1)^{a+b} q^{\left(p+\frac{1}{2}\right)a^2 + 2ab + \frac{3}{2} b^2 + \frac{2p+1}{2} a + \frac{5}{2} b}.
	\]
\end{theorem}

\begin{proof}
	By the definition,
	\begin{align*}
		H_p^{(1)}(q) = \sum_{s_p=0}^\infty q^{s_p} (q^{s_p+1})_{s_p+1} (q)_{s_p} \beta_{s_p}^{(2,p)} = \sum_{b=0}^\infty q^b (q)_{2b+1} \beta_b^{(2,p)}.
	\end{align*}
	Then, \cref{p-Bailey-result} implies that
	\begin{align*}
		H_p^{(1)}(q) &= \sum_{b=0}^\infty q^b (q)_{2b+1} \sum_{a=-b-1}^b \frac{(-1)^a q^{\left(p+\frac{1}{2}\right)a^2 + \left(p-\frac{1}{2}\right)a}}{(q)_{b-a} (q)_{b+a+1}}.
	\end{align*}
	By dividing the range of the inner sum into $0\leq a \leq b$ and $-b-1 \leq a \leq -1$ and changing the order of the sums,
	\begin{align*}
		H_p^{(1)} &=  \sum_{a\geq 0} (-1)^a q^{\left(p+\frac{1}{2}\right)a^2 + \left(p-\frac{1}{2}\right)a} \sum_{b=a}^\infty \frac{q^b (q)_{2b+1}}{(q)_{b-a} (q)_{b+a+1}}\\
			&\qquad + \sum_{a<0} (-1)^a q^{\left(p+\frac{1}{2}\right)a^2 + \left(p-\frac{1}{2}\right)a} \sum_{b=-a-1}^\infty \frac{q^b (q)_{2b+1}}{(q)_{b-a} (q)_{b+a+1}}.
	\end{align*}
	Changing the variables in two sums so that the range of $b$ is $b \geq 0$ implies
	\begin{align*}
		H_p^{(1)}(q) &= \sum_{a\geq 0} (-1)^a q^{\left(p+\frac{1}{2}\right)a^2 + \left(p-\frac{1}{2}\right)a} \sum_{b=0}^\infty \frac{q^{a+b} (q)_{2b+2a+1}}{(q)_b (q)_{b+2a+1}} \\
			&\qquad + \sum_{a<0} (-1)^a q^{\left(p+\frac{1}{2}\right)a^2 + \left(p-\frac{1}{2}\right)a} \sum_{b=0}^\infty \frac{q^{b-a-1} (q)_{2b-2a-1}}{(q)_b(q)_{b-2a-1}}.
	\end{align*}
	Finally, \cref{Fine} yields that
	\begin{align*}
		H_p^{(1)}(q) &= \frac{1}{(q)_\infty} \sum_{a\geq 0} (-1)^a q^{\left(p+\frac{1}{2}\right)a^2 + \left(p+\frac{1}{2}\right)a} \sum_{b=0}^\infty (-1)^b q^{\frac{3}{2} b^2 + \left(2a+\frac{5}{2} \right)b}\\
			&\qquad + \frac{1}{(q)_\infty} \sum_{a<0} (-1)^a q^{\left(p+\frac{1}{2}\right)a^2 + \left(p-\frac{3}{2}\right)a-1} \sum_{b=0}^\infty (-1)^b q^{\frac{3}{2}b^2 + \left(-2a+\frac{1}{2}\right)b}\\
			&= \frac{1}{(q)_\infty} \left(\sum_{a, b \geq 0} - \sum_{a,b < 0} \right) (-1)^{a+b} q^{\left(p+\frac{1}{2}\right) a^2 + 2ab + \frac{3}{2} b^2 + \left(p+\frac{1}{2}\right)a + \frac{5}{2} b},
	\end{align*}
	which is the desired Hecke-type expansion.
\end{proof}

\begin{theorem}[{Hikami~\cite[Theorem 3.9]{Hikami2007}}]\label{Habiro-5}
	For $p \geq 1$, we have the Hecke-type expansion for $H_p^{(5)}(q)$ as
	\[
		H_p^{(5)}(q) = \frac{1}{(q)_\infty} \left(\sum_{a, b \geq 0} - \sum_{a, b < 0} \right) (-1)^{a+b} q^{\left(p+\frac{1}{2}\right)a^2 + 2ab + \frac{3}{2} b^2 + \frac{1}{2} a + \frac{3}{2} b}.
	\]
\end{theorem}

\begin{proof}
	The idea of the proof is entirely the same as that of \cref{Hikami-Theorem3.5}. By the definition,
	\[
		H_p^{(5)}(q) = \sum_{s_p=0}^\infty q^{s_p} (q^{s_p+1})_{s_p} (q)_{s_p} \beta_{s_p}^{(1, p)} = \sum_{b=0}^\infty q^b (q)_{2b} \beta_b^{(1,p)}.
	\]
	Then, \cref{p-Bailey-result} implies that
	\[
		H_p^{(5)}(q) = \sum_{b=0}^\infty q^b (q)_{2b} \sum_{a=-b}^b \frac{(-1)^a q^{\left(p+\frac{1}{2} \right)a^2 - \frac{1}{2}a}}{(q)_{b-a} (q)_{b+a}}.
	\]
	By dividing the range of the sum into $0 \leq a \leq b$ and $-b \leq a \leq -1$, changing the order of the sums, and changing the variables in two sums so that the range of $b$ is $b \geq 0$, we have
	\begin{align*}
		H_p^{(5)}(q) &= \sum_{a \geq 0} (-1)^a q^{\left(p+\frac{1}{2} \right)a^2 + \frac{1}{2}a} \sum_{b=0}^\infty \frac{q^{b} (q)_{2b+2a}}{(q)_{b+2a} (q)_{b}}\\
			&\qquad + \sum_{a<0} (-1)^a q^{\left(p+\frac{1}{2}\right)a^2 - \frac{3}{2}a} \sum_{b=0}^\infty \frac{q^{b} (q)_{2b-2a}}{(q)_{b-2a} (q)_{b}}.
	\end{align*}
	Finally, \cref{Fine} yields the result.
\end{proof}

In the two cases above, the straightforward calculations transformed the Habiro-type series into a form in which \cref{Fine} can be applied. However, in the remaining three cases, additional modifications are required.

\begin{theorem}\label{Habiro-2}
	For $p \geq 1$, we have the Hecke-type expansion for $H_p^{(2)}(q)$ as
	\[
		H_p^{(2)}(q) = \frac{1}{(q)_\infty} \left(\sum_{a, b \geq 0} - \sum_{a, b < 0} \right) (-1)^{a+b} q^{\left(p+\frac{1}{2}\right)a^2 + 2ab + \frac{3}{2} b^2 + \frac{1}{2} a + \frac{1}{2} b}.
	\]
\end{theorem}

\begin{proof}
	By the definition, we have
	\[
		H_p^{(2)}(q) = \sum_{s_p=0}^\infty q^{s_p} (q^{s_p})_{s_p} (q)_{s_p} \beta_{s_p}^{(1,p)} = 1 + \sum_{b=1}^\infty q^{b} (q)_{2b-1} (1-q^{b}) \beta_{b}^{(1,p)}.
	\]
	By the auxiliary \cref{auxiliary1}, 
	\[
		H_p^{(2)}(q) = 1+\sum_{b=1}^\infty q^{b} (q)_{2b-1} \sum_{a=-b+1}^{b} \frac{(-1)^a q^{\left(p+\frac{1}{2}\right)a^2 - \frac{1}{2}a}}{(q)_{b-a} (q)_{b+a-1}}.
	\]
	We now divide the range of the sum into $a>0$, $a=0$, and $a<0$. Then the same transformation yields that
	\begin{align*}
		H_p^{(2)}(q) &= 1+ q \sum_{b=0}^\infty \frac{q^b (q)_{2b+1}}{(q)_b (q)_{b+1}} + \sum_{a>0} (-1)^a q^{\left(p+\frac{1}{2}\right)a^2 - \frac{1}{2}a} \sum_{b=a}^\infty \frac{q^{b} (q)_{2b-1}}{(q)_{b-a} (q)_{b+a-1}}\\
			&\qquad + \sum_{a<0} (-1)^a q^{\left(p+\frac{1}{2}\right)a^2 - \frac{1}{2}a} \sum_{b=-a+1}^\infty \frac{q^{b} (q)_{2b-1}}{(q)_{b-a} (q)_{b+a-1}}.
	\end{align*}
	By applying \cref{Fine}, 
	\begin{align*}
	H_p^{(2)}(q) &= 1+ \frac{1}{(q)_\infty} \sum_{b=0}^\infty (-1)^b q^{\frac{3}{2}b^2 +\frac{5}{2} b + 1} + \frac{1}{(q)_\infty} \sum_{a >0} (-1)^a q^{\left(p+\frac{1}{2}\right)a^2 + \frac{1}{2}a} \sum_{b=0}^\infty (-1)^b q^{\frac{3}{2} b^2 + \left(2a+\frac{1}{2} \right)b}\\
		&\qquad + \frac{1}{(q)_\infty} \sum_{a<0} (-1)^a q^{\left(p+\frac{1}{2}\right)a^2 - \frac{3}{2} a+1} \sum_{b=0}^\infty (-1)^b q^{\frac{3}{2}b^2 + \left(-2a+\frac{5}{2}\right)b}.
\end{align*}
	Finally, in Jacobi's triple product~\eqref{Jacobi-triple}, by setting $q \mapsto q^3$ and $\zeta \mapsto q$, we obtain Euler's pentagonal number theorem
	\begin{align}\label{Pentagon}
		(q)_\infty = \sum_{b \in \Z} (-1)^b q^{\frac{3}{2}b^2 - \frac{1}{2}b} = -\sum_{b \geq 0} (-1)^b q^{\frac{3}{2} b^2 + \frac{5}{2} b + 1} + \sum_{b \geq 0} (-1)^b q^{\frac{3}{2} b^2 + \frac{1}{2} b}.
	\end{align}
	Then we obtain the desired result by rearranging the equation.
\end{proof}

\begin{theorem}\label{Habiro-3}
	For $p \geq 1$, we have the Hecke-type expansion for $H_p^{(3)}(q)$ as
	\[
		H_p^{(3)}(q) = \frac{1}{(q)_\infty} \left(\sum_{a, b \geq 0} - \sum_{a, b < 0} \right) (-1)^{a+b} q^{\left(p+\frac{1}{2}\right)a^2 + 2ab + \frac{3}{2} b^2 + \frac{2p+3}{2} a + \frac{3}{2} b}.
	\]
\end{theorem}

\begin{proof}
	By the definition and \cref{auxiliary2}, we have
	\begin{align*}
		H_p^{(3)}(q) &= \sum_{s_p=0}^\infty q^{2s_p} (q^{s_p+1})_{s_p} (q)_{s_p} \beta_{s_p}^{(2, p)} = \sum_{b=0}^\infty q^b (q)_{2b} q^b \beta_b^{(2, p)}\\
		&= \sum_{b=0}^\infty q^b (q)_{2b} \sum_{a=-b}^b \frac{(-1)^a q^{\left(p+\frac{1}{2}\right)a^2+\left(p+\frac{1}{2}\right)a}}{(q)_{b-a} (q)_{b+a}}.
	\end{align*}
	The same calculation and \cref{Fine} yield that
	\begin{align*}
		H_p^{(3)}(q) &= \sum_{a\geq 0} (-1)^a q^{\left(p+\frac{1}{2}\right)a^2+\left(p+\frac{3}{2}\right)a} \sum_{b=0}^\infty \frac{q^{b} (q)_{2b+2a}}{(q)_{b+2a} (q)_{b}}\\
			&\qquad + \sum_{a<0} (-1)^a q^{\left(p+\frac{1}{2}\right)a^2+\left(p-\frac{1}{2}\right)a} \sum_{b=0}^\infty \frac{q^{b} (q)_{2b-2a}}{(q)_{b-2a} (q)_{b}}\\
			&= \frac{1}{(q)_\infty} \sum_{a\geq 0} (-1)^a q^{\left(p+\frac{1}{2}\right)a^2 + \left(p+\frac{3}{2}\right)a} \sum_{b=0}^\infty (-1)^b q^{\frac{3}{2} b^2 + \left(2a+\frac{3}{2} \right)b}\\
		&\qquad + \frac{1}{(q)_\infty} \sum_{a<0} (-1)^a q^{\left(p+\frac{1}{2}\right)a^2 + \left(p-\frac{1}{2}\right)a} \sum_{b=0}^\infty (-1)^b q^{\frac{3}{2}b^2 + \left(-2a+\frac{3}{2}\right)b},
	\end{align*}
	which finishes the proof.
\end{proof}

\begin{theorem}\label{Habiro-4}
	For $p \geq 1$, we have the Hecke-type expansion for $H_p^{(4)}(q)$ as
	\[
		H_p^{(4)}(q) = \frac{1}{(q)_\infty} \left(\sum_{a, b \geq 0} - \sum_{a, b < 0} \right) (-1)^{a+b} q^{\left(p+\frac{1}{2}\right)a^2 + 2ab + \frac{3}{2} b^2 + \frac{2p+1}{2} a + \frac{1}{2} b}.
	\]
\end{theorem}

\begin{proof}
	The structure of the proof is entirely the same as that for $H_p^{(2)}(q)$. First, by the definition and \cref{auxiliary3},
	\begin{align*}
		H_p^{(4)}(q) &= \sum_{s_p=0}^\infty q^{s_p} (q^{s_p+1})_{s_p} (q)_{s_p} \beta_{s_p}^{(2, p)} = 1 + \sum_{b=1}^\infty q^b (q)_{2b-1} (1-q^{2b}) \beta_b^{(2, p)}\\
			&= 1+ \sum_{b=1}^\infty q^b (q)_{2b-1} \sum_{a=-b+1}^b \frac{(-1)^a q^{\left(p+\frac{1}{2}\right)a^2+\left(p-\frac{1}{2}\right)a}}{(q)_{b-a} (q)_{b+a-1}}.
\end{align*}
	By dividing the range of the sum into $a>0$, $a=0$, and $a<0$, the equation can be transformed to
	\begin{align*}
		H_p^{(4)}(q) &= 1 + q \sum_{b=0}^\infty \frac{q^{b} (q)_{2b+1}}{(q)_{b+1} (q)_{b}} + \sum_{a>0} (-1)^a q^{\left(p+\frac{1}{2}\right)a^2+\left(p+\frac{1}{2}\right)a} \sum_{b=0}^\infty \frac{q^{b} (q)_{2b+2a-1}}{(q)_{b} (q)_{b+2a-1}}\\
			&\qquad + \sum_{a<0} (-1)^a q^{\left(p+\frac{1}{2}\right)a^2+\left(p-\frac{3}{2}\right)a+1} \sum_{b=0}^\infty  \frac{q^{b} (q)_{2b-2a+1}}{(q)_{b} (q)_{b-2a+1}}.
	\end{align*}
	Finally, \cref{Fine} and Euler's pentagonal number theorem~\eqref{Pentagon} yield the desired Hecke-type expansion.
\end{proof}

In conclusion, the goal of this section, which is to transform the five Habiro-type series into a Hecke-type series, has been achieved.

\subsection{Appendix on another conjugate Bailey pair and multiple zeta values}\label{s2-4}

Although we are in the middle of a discussion, we would like to make one more observation about a (conjugate) Bailey pair for another pair $(u, v)$. 

Here we consider the most straightforward pair $u_n = v_n = 1$ for any $n \geq 0$ instead of the $q$-Pochhammer symbol. Then using a similar approach to the Bailey transform, if sequences $(\alpha)_n$, $(\beta_n)_n$, $(\gamma_n)_n$, and $(\delta_n)_n$ satisfy suitable convergence conditions and the equation
\begin{align}\label{Bailey-11}
	\beta_n = \sum_{k=0}^{n-1} \alpha_k, \qquad \gamma_n = \sum_{k=n+1}^\infty \delta_k,
\end{align}
then we have
\begin{align}\label{trivial-Bailey}
	\sum_{n=0}^\infty \alpha_n \gamma_n = \sum_{n=1}^\infty \beta_n \delta_n.
\end{align}
In this setting, we try finding an interesting (conjugate) Bailey pair.

\begin{lemma}
	For any positive integer $m > 0$, the pair of sequences defined by
	\[
		\gamma_n = \frac{1}{m} \frac{n! m!}{(n+m)!}, \qquad \delta_n = \frac{1}{n} \frac{n! m!}{(n+m)!}
	\]
	satisfies the condition in~\eqref{Bailey-11}.
\end{lemma}

\begin{proof}
	By a telescoping sum, we have
	\[
		\sum_{k=n+1}^\infty \delta_k = \frac{1}{m} \sum_{k=n+1}^\infty \left(\frac{(k-1)! m!}{(k-1+m)!} - \frac{k! m!}{(k+m)!} \right) = \frac{1}{m} \frac{n! m!}{(n+m)!} = \gamma_n.
	\]
\end{proof}

This lemma is a critical identity in Seki--Yamamoto's work~\cite{SekiYamamoto2019} on the proof of duality of multiple zeta values. What is the ``unit Bailey pair/chain" in this case? A simple observation yields that if a pair $(\alpha, \beta)$ satisfies the condition~\eqref{Bailey-11}, the equation \eqref{trivial-Bailey} implies an identity on the sequence $(\alpha_n)_n$,
\begin{align}\label{alpha-eq}
	\sum_{n=0}^\infty \frac{\alpha_n}{m} \frac{n! m!}{(n+m)!} = \sum_{n=1}^\infty \frac{\beta_n}{n} \frac{n! m!}{(n+m)!} = \sum_{n = 1}^\infty \left(\sum_{k=0}^{n-1} \frac{\alpha_k}{n} \right) \frac{n! m!}{(n+m)!}.
\end{align}
Putting $\alpha^{(1)}_0 = 0$ and $\alpha^{(1)}_n = \sum_{k=0}^{n-1} \alpha_k/n$ and applying \eqref{alpha-eq}, we have
\[
	\sum_{n=0}^\infty \frac{\alpha_n^{(1)}}{m} \frac{n! m!}{(n+m)!} = \sum_{n=1}^\infty \left(\sum_{k=0}^{n-1} \frac{\alpha_k^{(1)}}{n} \right) \frac{n! m!}{(n+m)!} = \sum_{n=1}^\infty \left(\sum_{n > k > l \geq 0} \frac{\alpha_l}{kn} \right) \frac{n! m!}{(n+m)!}.
\]
By repeating this process, we obtain the following.
\begin{lemma}
	For any sequence $(\alpha_n)_n$ with suitable convergence condition, it holds that
	\[
		\frac{1}{m^p} \sum_{n=0}^\infty \alpha_n \frac{n!m!}{(n+m)!} = \sum_{n=0}^\infty \sum_{n = s_p > s_{p-1} > \cdots > s_0 \geq 0} \frac{\alpha_{s_0}}{s_p s_{p-1} \cdots s_1} \frac{n! m!}{(n+m)!}.
	\]
\end{lemma}

We now let $\alpha_n = \delta_{n,0}$. Then the above lemma yields that
\[
	\frac{1}{m^p} = \sum_{s_p > s_{p-1} > \cdots > s_1 > 0} \frac{1}{s_p s_{p-1} \cdots s_1} \frac{s_p! m!}{(s_p+m)!}
\]
for any $p > 0$. Moreover, by taking the sum over $m$, we have
\begin{align*}
	\zeta(p+1) &= \sum_{m=1}^\infty \frac{1}{m^{p+1}} = \sum_{s_p > s_{p-1} > \cdots > s_1 > 0} \frac{1}{s_{p-1} \cdots s_1} \sum_{m=1}^\infty \frac{(s_p-1)!(m-1)!}{(s_p+m)!}\\
	&= \sum_{s_p > s_{p-1} > \cdots > s_1 > 0} \frac{1}{s_p^2 s_{p-1} \cdots s_1} = \zeta(2, \underbrace{1, \dots, 1}_{p-1}),
\end{align*}
the particular case of the sum formula of multiple zeta values~\cite{Granville1997}. We also note that the factor $(n-1)! (m-1)!/(n+m)!$ appears in Cloitre and Oloa's expressions of the Riemann zeta values $\zeta(2)$ and $\zeta(3)$ (see Kawamura--Maesaka--Seki's recent work~\cite{KawamuraMaesakaSeki2022}). It will be interesting to see what kind of identities can be obtained by considering various pairs $(u,v)$, $(\gamma, \delta)$, and $(\alpha, \beta)$.

\section{False theta functions}\label{s3}

False theta functions also come from the work of Rogers~\cite{Rogers1917London}. The defining equation looks like the ordinary theta functions but contains an extra sign term that breaks the modular property of the theta functions. The name ``false theta functions" also appeared in Ramanujan's last letter to Hardy, and it became known at about the same time as Ramanujan's ``mock theta functions". However, as Sills~\cite{Sills2018} points out, mock theta functions have long been the subject of active research, starting with Watson~\cite{Watson1936}, whereas false theta functions have not received much attention until recently. One of the reasons why the study of false theta functions has become active in recent years is due to their relationship with quantum invariants, as revealed by Lawrence--Zagier~\cite{LawrenceZagier1999}. Their work was subsequently generalized extensively by Hikami~\cite{Hikami2005IJM, Hikami2006JMP, Hikami2007} et al., and the advent of the notion of ``quantum modular forms" introduced by Zagier~\cite{Zagier2010} has further accelerated the research. More recently, as a counterpart to Zwegers' discovery~\cite{Zwegers2002} of the modular aspect of the mock theta functions, a framework for modular properties of false theta functions has been revealed by Bringmann--Nazaroglu et al.~\cite{BringmannNazaroglu2019, BKMN2023} and Goswami--Osburn~\cite{GoswamiOsburn2021} (see also Matsusaka--Terashima~\cite{MatsusakaTerashima2021}). 

Let us define the false theta functions. Let $A$ be a symmetric $r \times r$-matrix with integer coefficients, which is positive definite. We consider a bilinear form $B: \R^r \times \R^r \to \R$, $B(\bm{x}, \bm{y}) = \bm{x}^T A \bm{y}$ and the associated quadratic form $Q(\bm{x}) = \frac{1}{2} B(\bm{x}, \bm{x})$. We take a lattice $L \subset \R^r$ of rank $r$ such that $B(L \times L) \subset \Z$. Then the \emph{dual lattice} $L^*$ is defined by
\[
	L^* = \{\bm{x} \in \R^r \mid B(\bm{x}, \bm{y}) \in \Z \text{ for any $\bm{y} \in L$}\}.
\]
\begin{definition}\label{def-false-theta}
	For an arbitrary vector $\bm{c} \in \R^r$ satisfying $2Q(\bm{c}) = 1$ and $\bm{\mu} \in L^*$, the \emph{false theta function} $\widetilde{\theta}_{\bm{\mu}, \bm{c}}(\tau) = \widetilde{\theta}_{L, B, \bm{\mu}, \bm{c}}(\tau)$ is defined by
	\[
		\widetilde{\theta}_{\bm{\mu}, \bm{c}}(\tau) = \sum_{\bm{n} \in L + \bm{\mu}} \sgn(B(\bm{n}, \bm{c})) q^{Q(\bm{n})},
	\]
	where $\tau \in \bbH = \{\tau \in \C \mid \Im(\tau) > 0\}$, $q^{Q(\bm{n})} = e^{2\pi i Q(\bm{n}) \tau}$, and $\sgn: \R \to \{-1, 0, 1\}$ is the usual sign function with $\sgn(0) = 0$.
\end{definition}

This article considers only the cases $r=1, 2$ and only special lattices and bilinear forms.

\subsection{One dimensional case}

This subsection aims to recall Hikami's function $\widetilde{\Phi}_{\bm{p}}^{\bm{\ell}}(\tau)$ in terms of false theta functions and re-prove \cref{Hikami-Prop3}. We let $A=(1)$ and $L = \sqrt{M} \Z$ for a positive integer $M > 0$. Then the bilinear form is $B(x,y) = xy$, and the dual lattice is given by $L^* = (1/\sqrt{M}) \Z$. In this setting, the false theta function is defined as
\[
	\widetilde{\theta}_{\mu, c}(\tau) = \sum_{n \in \sqrt{M} \Z + \mu} \sgn(n) q^{\frac{n^2}{2}}
\]
for $\mu \in (1/\sqrt{M})\Z/\sqrt{M}\Z \cong \Z/M\Z$ and $c = 1$. To clarify necessary and unnecessary subscripts, we redefine the false theta function.

\begin{definition}
	For a positive integer $M > 0$ and $\mu \in \{0, 1, \dots, M-1\}$, we define
	\[
		\widetilde{\theta}_{M, \mu}^{(1)}(\tau) = \sum_{n \equiv \mu\ (M)} \sgn(n) q^{\frac{n^2}{2M}}.
	\]
\end{definition}

Although not directly used in this article, the following modular transformation formulas are known by Bringmann--Nazaroglu~\cite{BringmannNazaroglu2019} (see also~\cite{MatsusakaTerashima2021}).
\begin{proposition}\label{S-trans-ft1}
	For $\Re(\tau) \neq 0$, we have
	\begin{align*}
		&\widetilde{\theta}_{M, \mu}^{(1)} \left(-\frac{1}{\tau} \right) + \sgn(\Re(\tau)) \frac{(-i\tau)^{1/2}}{\sqrt{M}} \sum_{\nu =0}^{M-1} e^{2\pi i \frac{\mu \nu}{M}} \widetilde{\theta}_{M, \nu}^{(1)} (\tau) = \frac{-i}{\sqrt{M}} \int_0^{i\infty} \frac{\vartheta_{M,\mu}^{(1)} (z)}{\sqrt{-i (z+1/\tau)}} dz,
	\end{align*} 
	where $\vartheta_{M,\mu}^{(1)}(\tau)$ is the classical holomorphic theta function of weight $3/2$ defined by
	\[
		\vartheta_{M, \mu}^{(1)}(\tau) = \sum_{n \equiv \mu\ (M)} n q^{\frac{n^2}{2M}}.
	\]
\end{proposition}

In the above transformation formula, a remarkable feature of the false theta functions is that an error term expressed by the integral of a modular form appears. A similar phenomenon is observed in the transformation formula for the mock theta functions presented by Watson~\cite{Watson1936}. This similarity may be a part of the ``mock vs. false" phenomenon discussed by Lawrence--Zagier~\cite{LawrenceZagier1999}, Zwegers~\cite[Conjecture 2.2]{Zwegers2001}, and Hikami~\cite{Hikami2005RCD} et al., but many mysteries remain.

Since we will use the transformation formulas for the ordinary theta functions later, we review the definition and the claim together.

\begin{lemma}\label{theta-modular}
	For a positive even integer $M > 0$ and $\mu \in \{0, 1, \dots, M-1\}$, the ordinary theta function defined by
	\[
		\theta_{M,\mu}^{(1)}(\tau) = \sum_{n \equiv \mu\ (M)} q^{\frac{n^2}{2M}}
	\]
	satisfies
	\begin{align*}
		\theta_{M,\mu}^{(1)} (\tau+1) &= e^{2\pi i \frac{\mu^2}{2M}} \theta_{M, \mu}^{(1)}(\tau),\\
		\theta_{M,\mu}^{(1)} \left(-\frac{1}{\tau} \right) &= \frac{(-i\tau)^{1/2}}{\sqrt{M}} \sum_{\nu=0}^{M-1} e^{2\pi i \frac{\mu \nu}{M}} \theta_{M,\nu}^{(1)}(\tau).
	\end{align*}
\end{lemma}

We will now rewrite the function $\widetilde{\Phi}_{(p_1, p_2, p_3)}^{(\ell_1, \ell_2, \ell_3)} (\tau)$ that Hikami considered in~\cite{Hikami2005IJM, Hikami2007} in terms of the false theta functions we have just prepared. Let $\bm{p} = (p_1, p_2, p_3)$ be a triple of pairwise coprime positive integers, and $P = p_1 p_2 p_3$. For any triple $\bm{\ell} = (\ell_1, \ell_2, \ell_3) \in \Z^3$ satisfying $0 < \ell_j < p_j$, we define an odd periodic function $\chi_{\bm{p}}^{\bm{\ell}}: \Z/2P\Z \to \Z$ by
\begin{align}\label{period-chi}
	\chi_{\bm{p}}^{\bm{\ell}} (n) = \begin{cases}
		-\varepsilon_1 \varepsilon_2 \varepsilon_3 &\text{if } n \equiv P \left(1 + \sum_{j=1}^3 \frac{\varepsilon_j \ell_j}{p_j} \right) \pmod{2P},\\
		0 &\text{if otherwise},
	\end{cases}
\end{align}
where $\bm{\varepsilon} = (\varepsilon_1, \varepsilon_2, \varepsilon_3) \in \{\pm 1\}^3$. Then the function $\widetilde{\Phi}_{\bm{p}}^{\bm{\ell}}(\tau)$ is defined by
\[
	\widetilde{\Phi}_{\bm{p}}^{\bm{\ell}} (\tau) = \sum_{n=0}^\infty \chi_{\bm{p}}^{\bm{\ell}}(n) q^{\frac{n^2}{4P}}.
\]

\begin{lemma}\label{Eichler-false}
	The function $\widetilde{\Phi}_{\bm{p}}^{\bm{\ell}} (\tau)$ is expressed in terms of false theta functions as
	\[
		\widetilde{\Phi}_{\bm{p}}^{\bm{\ell}} (\tau) = -\frac{1}{2} \sum_{\bm{\varepsilon} \in \{\pm 1\}^3} \varepsilon_1 \varepsilon_2 \varepsilon_3 \widetilde{\theta}_{2P, \mu(\bm{\varepsilon}, \bm{\ell})}^{(1)}(\tau),
	\]
	where we put
	\[
		\mu(\bm{\varepsilon}, \bm{\ell}) \equiv P \left(1 + \sum_{j=1}^3 \frac{\varepsilon_j \ell_j}{p_j} \right) \pmod{2P}.
	\]
\end{lemma}

\begin{proof}
	Since $\chi_{\bm{p}}^{\bm{\ell}}(-n) = -\chi_{\bm{p}}^{\bm{\ell}}(n)$ holds for any $n \in \Z$, we have
	\begin{align*}
		\widetilde{\Phi}_{\bm{p}}^{\bm{\ell}} (\tau) &= \frac{1}{2} \sum_{n \in \Z} \sgn(n) \chi_{\bm{p}}^{\bm{\ell}}(n) q^{\frac{n^2}{4P}}\\
			&= \frac{1}{2} \sum_{\bm{\varepsilon} \in \{\pm 1\}^3} \sum_{n \equiv \mu(\bm{\varepsilon}, \bm{\ell})\ (2P)} \sgn(n) (-\varepsilon_1 \varepsilon_2 \varepsilon_3)  q^{\frac{n^2}{4P}},
	\end{align*}
	which concludes the proof.
\end{proof}

In this setting, we re-prove the limit formula of $\widetilde{\Phi}_{\bm{p}}^{\bm{\ell}} (\tau)$ as $\tau \to 1/N$ shown by Hikami~\cite[Proposition 3]{Hikami2005IJM}. To this end, we recall the following two analytic lemmas by Lawrence--Zagier~\cite{LawrenceZagier1999}.

\begin{lemma}
	Let $C: \Z \to \C$ be a periodic function whose period is $M$. If its mean value equals $0$, that is,
	\[
		\sum_{m=1}^M C(m) = 0,
	\]
	then the Dirichlet series $L(s, C) = \sum_{m=1}^\infty C(m)m^{-s}$ defines a holomorphic function in $\Re(s) > 1$ and is analytically continued to the whole $\C$-plane. The special values at negative integers satisfy
	\[
		L(-r, C) = -\frac{M^r}{r+1} \sum_{m=1}^M C(m) B_{r+1} \left(\frac{m}{M} \right),
	\] 
	where $B_m(x)$ is the $m$-th Bernoulli polynomial defined by
	\[
		\sum_{m=0}^\infty B_m(x) \frac{t^m}{m!} = \frac{t e^{xt}}{e^t-1}.
	\]
\end{lemma}

\begin{lemma}
	The following asymptotic expansion holds.
	\[
		\sum_{m=1}^\infty C(m) e^{-m^2 t} \sim \sum_{m=0}^\infty L(-2m, C) \frac{(-t)^m}{m!} \quad (t \to 0+).
	\]
\end{lemma}

\begin{proposition}\label{false-limit}
	For positive integers $M,N > 0$ and $\mu \in \{0, 1, \dots, M-1\}$ with $2\mu \not\equiv 0 \pmod{M}$, we have
	\[
		\lim_{t \to 0+} \widetilde{\theta}_{M, \mu}^{(1)} \left(\frac{1}{N} + it \right) = - \frac{1}{2MN} \sum_{n=1}^{2MN} n \psi_{M, \mu}(n) e^{\pi i \frac{n^2}{MN}},
	\]
	where we put
	\[
		\psi_{M, \mu}(n) = \begin{cases}
			\varepsilon &\text{if } n \equiv \varepsilon \mu \pmod{M},\\
			0 &\text{if otherwise},
		\end{cases}
	\]
	for $\varepsilon \in \{\pm 1\}$.
\end{proposition}

\begin{proof}
	By the definition, we have
	\begin{align*}
		\widetilde{\theta}_{M, \mu}^{(1)} \left(\frac{1}{N} + it \right) &= \sum_{n \equiv \mu\ (M)} \sgn(n) e^{\pi i \frac{n^2}{MN}} e^{-n^2 \frac{\pi t}{M}} = \sum_{n = 1}^\infty \psi_{M,\mu}(n) e^{\pi i \frac{n^2}{MN}} e^{-n^2 \frac{\pi t}{M}}.
	\end{align*}
	Since the function $C_{M, \mu, N}(n) =\psi_{M,\mu}(n) e^{\pi i \frac{n^2}{MN}}$ has a period $2MN$ and its mean value equals $0$, we can apply Lawrence--Zagier's lemmas. Then we have
	\begin{align*}
		\lim_{t \to 0+} \widetilde{\theta}_{M, \mu}^{(1)} \left(\frac{1}{N} + it \right) &= L(0, C_{M, \mu, N}) = -\sum_{n=1}^{2MN} C_{M, \mu, N}(n) B_1 \left(\frac{n}{2MN}\right).
	\end{align*}
	The fact that $B_1(x) = x - 1/2$ concludes the proof.
\end{proof}

\begin{corollary}[{\cite[Proposition 3]{Hikami2005IJM}}]\label{Hikami-Prop3}
	\[
		\lim_{t \to 0+} \widetilde{\Phi}_{\bm{p}}^{\bm{\ell}} \left(\frac{1}{N} + it \right) = -\frac{1}{2PN} \sum_{n=1}^{2PN} n \chi_{\bm{p}}^{\bm{\ell}}(n) e^{\pi i \frac{n^2}{2PN}}.
	\]
\end{corollary}

\begin{proof}
	By \cref{Eichler-false} and \cref{false-limit}, we have
	\[
		\lim_{t \to 0+} \widetilde{\Phi}_{\bm{p}}^{\bm{\ell}} \left(\frac{1}{N} + it \right) = \frac{1}{8PN} \sum_{\bm{\varepsilon} \in \{\pm 1\}^3} \varepsilon_1 \varepsilon_2 \varepsilon_3 \sum_{n=1}^{4PN} n \psi_{2P, \mu(\bm{\varepsilon}, \bm{\ell})}(n) e^{\pi i \frac{n^2}{2PN}}.
	\]
	First, since $\psi_{2P, \mu(\bm{\varepsilon}, \bm{\ell})}(n) e^{\pi i \frac{n^2}{2PN}}$ has a period of $2PN$ and its mean value equals $0$, the inner sum is reduced to
	\begin{align*}
		\sum_{n=1}^{4PN} n \psi_{2P, \mu(\bm{\varepsilon}, \bm{\ell})}(n) e^{\pi i \frac{n^2}{2PN}} &= \sum_{n=1}^{2PN} n \psi_{2P, \mu(\bm{\varepsilon}, \bm{\ell})}(n) e^{\pi i \frac{n^2}{2PN}} + \sum_{n=1}^{2PN} (2PN+n) \psi_{2P, \mu(\bm{\varepsilon}, \bm{\ell})}(n) e^{\pi i \frac{n^2}{2PN}}\\
			&= 2 \sum_{n=1}^{2PN} n \psi_{2P, \mu(\bm{\varepsilon}, \bm{\ell})}(n) e^{\pi i \frac{n^2}{2PN}}.
	\end{align*}
	Second, we can check that
	\[
		\sum_{\bm{\varepsilon} \in \{\pm 1\}^3} \varepsilon_1 \varepsilon_2 \varepsilon_3 \psi_{2P, \mu(\bm{\varepsilon}, \bm{\ell})}(n) = -2 \chi_{\bm{p}}^{\bm{\ell}}(n),
	\]
	which finishes the proof.
\end{proof}

\subsection{False theta decompositions in two dimensional case}\label{s3-2}

In the case of $r=2$, we consider $L = 2\Z^2$ and a bilinear form associated with the matrix
\[
	A = \pmat{2p+1 & 2 \\ 2 & 3}
\]
for a positive integer $p \geq 1$. In other words, $B(\bm{x}, \bm{y}) = \bm{x}^T A \bm{y}$ and $Q(\bm{x}) = (p+1/2) x_1^2 + 2 x_1 x_2 + 3/2 x_2^2$. The dual lattice is given by $L^* = \frac{1}{2} A^{-1} \Z^2$. As a vector $\bm{c} \in \R^2$ satisfying $2Q(\bm{c}) = 1$, we choose here
\[
	\bm{c_1} = \frac{1}{\sqrt{3(6p-1)}} \pmat{3 \\ -2}, \quad \bm{c_2} = \frac{1}{\sqrt{(2p+1)(6p-1)}} \pmat{-2 \\ 2p+1}.
\]
Then for $\bm{\mu} \in L^*$ and $\bm{c} \in \{\bm{c_1}, \bm{c_2}\}$, we consider the false theta function
\begin{align}\label{2-false-def}
	\widetilde{\theta}_{\bm{\mu}, \bm{c}}^{(2)}(\tau) = \sum_{\bm{n} \in L + \bm{\mu}} \sgn(B(\bm{n}, \bm{c})) q^{Q(\bm{n})}
\end{align}
defined in \cref{def-false-theta}. To distinguish it from the one-dimensional case, we intentionally add superscript here. In the next subsection, we reformulate the Hecke-type series shown in \cref{s2-3} using the above false theta functions. Before we do so, we prepare the decomposition formula to compute the limit values of false theta functions.

\begin{lemma}\label{heihoukansei}
	We define $Q_{\bm{c}}(\bm{x}) = Q(\bm{x}) - \frac{1}{2} B(\bm{x}, \bm{c})^2$, $\bm{e}_1 = \smat{1 \\ 0}$, and $\bm{e}_2 = \smat{0 \\ 1}$. For $\bm{x} = (x,y)$, we have
	\begin{align*}
		Q_{\bm{c_1}}(\bm{x}) &= \frac{1}{6} (2x+3y)^2 = \frac{1}{6} B(\bm{e}_2, \bm{x})^2,\\
		Q_{\bm{c_2}}(\bm{x}) &= \frac{1}{2(2p+1)} ((2p+1)x+2y)^2 = \frac{1}{2(2p+1)} B(\bm{e}_1, \bm{x})^2
	\end{align*}
	and
	\begin{align*}
		B(\bm{x}, \bm{c_1}) = \sqrt{\frac{6p-1}{3}} x, \qquad B(\bm{x}, \bm{c_2}) = \sqrt{\frac{6p-1}{2p+1}} y.
	\end{align*}
\end{lemma}

\begin{proof}
	It follows immediately from a direct calculation.
\end{proof}

\begin{lemma}\label{image-Qc}
	For $\bm{c} \in \{\bm{c_1}, \bm{c_2}\}$ and $\bm{\mu} = \frac{1}{2} A^{-1} \smat{m_1 \\ m_2} \in L^*$ with $\smat{m_1 \\ m_2} \in \Z^2$, we have
	\[
		Q_{\bm{c_1}}(L + \bm{\mu}) = \frac{1}{24} (4\Z + m_2)^2, \qquad Q_{\bm{c_2}}(L+ \bm{\mu}) = \frac{1}{8(2p+1)} (4\Z+m_1)^2.
	\]
\end{lemma}

\begin{proof}
	\cref{heihoukansei} and the equations
	\begin{align*}
		B(\bm{e}_2, L+\bm{\mu}) &= B(\bm{e}_2, L) + \frac{1}{2} \bm{e}_2^T A A^{-1} \pmat{m_1 \\ m_2} = 2\Z + \frac{m_2}{2},\\
		B(\bm{e}_1, L+\bm{\mu}) &= 2\Z + \frac{m_1}{2}
	\end{align*}
	yield the result.
\end{proof}

With these preparations, the false theta functions $\widetilde{\theta}_{\bm{\mu}, \bm{c}}^{(2)}(\tau)$ can be decomposed into a sum of products of the one-dimensional false theta functions and the ordinary theta functions.

\begin{theorem}\label{false-decompose1}
	For $\bm{\mu} = \frac{1}{2}A^{-1} \smat{m_1 \\ m_2} \in L^*$ with odd integers $m_1, m_2$, we have
	\[
		\widetilde{\theta}_{\bm{\mu}, \bm{c_1}}^{(2)}(\tau) = \sum_{j=0}^2 \theta_{12, m_2+4j}^{(1)}(\tau) \widetilde{\theta}_{12(6p-1), 3m_1-2m_2-4(6p-1)j}^{(1)}(\tau).
	\]
\end{theorem}

\begin{proof}
	By \cref{image-Qc},
	\begin{align*}
		\widetilde{\theta}_{\bm{\mu}, \bm{c_1}}^{(2)}(\tau) &= \sum_{\bm{n} \in L + \bm{\mu}} \sgn(B(\bm{n}, \bm{c_1})) q^{\frac{1}{2} B(\bm{n}, \bm{c_1})^2} q^{Q_{\bm{c_1}}(\bm{n})}\\
			&= \sum_{N \in \Z} q^{\frac{(4N+m_2)^2}{24}} \sum_{\substack{\bm{n} \in 2\Z^2 \\ Q_{\bm{c_1}}(\bm{n}+ \bm{\mu}) = \frac{(4N+m_2)^2}{24}}} \sgn(B(\bm{n}+\bm{\mu}, \bm{c_1})) q^{\frac{1}{2} B(\bm{n}+ \bm{\mu}, \bm{c_1})^2}.
	\end{align*}
	By \cref{heihoukansei}, the condition $Q_{\bm{c_1}}(\bm{n}+\bm{\mu}) = \frac{(4N+m_2)^2}{24}$ for $\bm{n} = 2\smat{n_1 \\ n_2} \in 2\Z^2$ is reduced to
	\begin{align*}
		2B(\bm{e}_2, \bm{n}+\bm{\mu}) = \pm (4N+m_2),
	\end{align*}
	that is, $2n_1+3n_2 = N$ or $2n_1 + 3n_2 = -N - \frac{m_2}{2}$. However, since $m_2$ is odd, the second possibility is eliminated. Therefore, we have
	\[
		\pmat{n_1 \\ n_2} = n \pmat{3 \\ -2} + N \pmat{-1 \\ 1}
	\]
	for any $n \in \Z$. For this $\bm{n}$, the value of the bilinear form is given by
	\begin{align*}
		B(\bm{n} + \bm{\mu}, \bm{c_1}) &= \sqrt{\frac{6p-1}{3}} \left(6n - 2N + \frac{3m_1 - 2m_2}{2(6p-1)}\right)\\
			&= \frac{12(6p-1) n + (3m_1 - 2m_2 - 4(6p-1)N)}{\sqrt{12(6p-1)}}.
	\end{align*}
	Therefore, we have
	\begin{align*}
		\widetilde{\theta}_{\bm{\mu}, \bm{c_1}}^{(2)} (\tau) &= \sum_{N \in \Z} q^{\frac{(4N+m_2)^2}{24}} \sum_{n \equiv 3m_1 - 2m_2 - 4(6p-1)N \pmod{12(6p-1)}} \sgn(n) q^{\frac{n^2}{24(6p-1)}}.
	\end{align*}
	If we classify $N$ modulo $3$, we conclude the claim.
\end{proof}

\begin{theorem}\label{false-decompose2}
	For $\bm{\mu} = \frac{1}{2}A^{-1} \smat{m_1 \\ m_2} \in L^*$ with odd integers $m_1, m_2$, we have
	\[
		\widetilde{\theta}_{\bm{\mu}, \bm{c_2}}^{(2)}(\tau) = \sum_{j=0}^{2p} \theta_{4(2p+1), m_1+4j}^{(1)}(\tau) \widetilde{\theta}_{4(2p+1)(6p-1), -2m_1+(2p+1)m_2-4p(6p-1)j}^{(1)}(\tau).
	\]
\end{theorem}

\begin{proof}
	The idea of the proof is entirely the same as that of \cref{false-decompose1}. By \cref{image-Qc},
	\begin{align*}
		\widetilde{\theta}_{\bm{\mu}, \bm{c_2}}^{(2)}(\tau) &= \sum_{\bm{n} \in L + \bm{\mu}} \sgn(B(\bm{n}, \bm{c_2})) q^{\frac{1}{2} B(\bm{n}, \bm{c_2})^2} q^{Q_{\bm{c_2}}(\bm{n})}\\
			&= \sum_{N \in \Z} q^{\frac{(4N+m_1)^2}{8(2p+1)}} \sum_{\substack{\bm{n} \in 2\Z^2 \\ Q_{\bm{c_2}}(\bm{n}+ \bm{\mu}) = \frac{(4N+m_1)^2}{8(2p+1)}}} \sgn(B(\bm{n}+\bm{\mu}, \bm{c_2})) q^{\frac{1}{2} B(\bm{n}+ \bm{\mu}, \bm{c_2})^2}.
	\end{align*}
	By \cref{heihoukansei}, the condition $Q_{\bm{c_2}}(\bm{n}+\bm{\mu}) = \frac{(4N+m_1)^2}{8(2p+1)}$ for $\bm{n} = 2\smat{n_1 \\ n_2} \in 2\Z^2$ is reduced to
	\begin{align*}
		2B(\bm{e}_1, \bm{n}+\bm{\mu}) = \pm (4N+m_1),
	\end{align*}
	that is, $(2p+1)n_1+2n_2 = N$ or $(2p+1)n_1 + 2n_2 = -N - \frac{m_1}{2}$. However, since $m_1$ is odd, the second possibility is eliminated. Therefore, we have
	\[
		\pmat{n_1 \\ n_2} = n \pmat{-2 \\ 2p+1} + N \pmat{1 \\ -p}
	\]
	for any $n \in \Z$. For this $\bm{n}$, the value of the bilinear form is given by
	\begin{align*}
		B(\bm{n} + \bm{\mu}, \bm{c_2}) &= \sqrt{\frac{6p-1}{2p+1}} \left(2(2p+1)n - 2p N + \frac{-2m_1 + (2p+1)m_2}{2(6p-1)}\right)\\
			&= \frac{4(2p+1)(6p-1) n + (-2m_1 + (2p+1)m_2 - 4p(6p-1)N)}{\sqrt{4(2p+1)(6p-1)}}.
	\end{align*}
	Therefore, we have
	\begin{align*}
		\widetilde{\theta}_{\bm{\mu}, \bm{c_2}}^{(2)} (\tau) &= \sum_{N \in \Z} q^{\frac{(4N+m_1)^2}{8(2p+1)}} \sum_{n \equiv -2m_1 + (2p+1)m_2 - 4p(6p-1)N \pmod{4(2p+1)(6p-1)}} \sgn(n) q^{\frac{n^2}{8(2p+1)(6p-1)}}.
	\end{align*}
	If we classify $N$ modulo $2p+1$, we conclude the claim.
\end{proof}

\subsection{Limit values}

Under the decomposition formulas given in \cref{false-decompose1} and \cref{false-decompose2}, we compute the limit values of Habiro-type series as $q \to e^{2\pi i/N}$ radially from within the unit disc. First, we transform Hecke-type series into false theta functions. The notations are the same as in \cref{s3-2}.

\begin{lemma}\label{Hecke-false}
	For any pair of integers $(m_1, m_2) \in \Z^2$, we put
	\[
		\bm{\mu} = \pmat{\mu_1 \\ \mu_2} := \frac{1}{2} A^{-1} \pmat{m_1 \\ m_2}.
	\]
	If $0 < \mu_1 < 1$ and $0 < \mu_2 < 1$, we have
	\begin{align*}
		&\frac{1}{(q)_\infty} \left(\sum_{a, b \geq 0} - \sum_{a, b < 0} \right) (-1)^{a+b} q^{\left(p+\frac{1}{2}\right) a^2 + 2ab + \frac{3}{2} b^2 + \frac{m_1}{2} a + \frac{m_2}{2} b}\\
			&= \frac{q^{\frac{1}{24} - Q(\bm{\mu})}}{2\eta(\tau)} \sum_{\bm{\varepsilon} = \smat{\varepsilon_1 \\ \varepsilon_2} \in \{0,1\}^2} (-1)^{\varepsilon_1 + \varepsilon_2} \bigg(\widetilde{\theta}_{\bm{\mu}+\bm{\varepsilon}, \bm{c_1}}^{(2)}(\tau) + \widetilde{\theta}_{\bm{\mu}+\bm{\varepsilon}, \bm{c_2}}^{(2)}(\tau) \bigg),
	\end{align*}
	where $\eta(\tau) = q^{1/24} (q)_\infty$ is the Dedekind eta function.
\end{lemma}

\begin{proof}
	By the definition,
	\begin{align*}
		&\sum_{\bm{\varepsilon} \in \{0,1\}^2} (-1)^{\varepsilon_1 + \varepsilon_2} \bigg(\widetilde{\theta}_{\bm{\mu}+\bm{\varepsilon}, \bm{c_1}}^{(2)}(\tau) + \widetilde{\theta}_{\bm{\mu}+\bm{\varepsilon}, \bm{c_2}}^{(2)}(\tau) \bigg)\\
			&= \sum_{\bm{\varepsilon} \in \{0,1\}^2} (-1)^{\varepsilon_1 + \varepsilon_2} \sum_{\bm{n} \in 2\Z^2} (\sgn(B(\bm{n}+ \bm{\mu} + \bm{\varepsilon}, \bm{c_1})) + \sgn(B(\bm{n}+ \bm{\mu} + \bm{\varepsilon}, \bm{c_2}))) q^{Q(\bm{n}+ \bm{\mu} + \bm{\varepsilon})}.
	\end{align*}
	Since $(-1)^{\varepsilon_1 + \varepsilon_2} = (-1)^{(n_1+\varepsilon_1) + (n_2+\varepsilon_2)}$ for any $\bm{n} \in 2\Z^2$, the above is equal to
	\begin{align*}
		\sum_{\bm{n} \in \Z^2} (-1)^{n_1+n_2} (\sgn(B(\bm{n}+ \bm{\mu}, \bm{c_1})) + \sgn(B(\bm{n}+ \bm{\mu}, \bm{c_2}))) q^{Q(\bm{n}+ \bm{\mu})}.
	\end{align*}
	By \cref{heihoukansei} and our assumption $0 < \mu_1< 1$, we have
	\begin{align*}
		\sgn(B(\bm{n}+ \bm{\mu}, \bm{c_1})) = \sgn(n_1+\mu_1) = \sgn_0(n_1),
	\end{align*}
	where $\sgn_0$ is the sign function with $\sgn_0(0) = 1$. Similarly $\sgn(B(\bm{n}+ \bm{\mu}, \bm{c_2})) = \sgn_0(n_2)$ holds. Thus we obtain
	\begin{align*}
		\sum_{\bm{\varepsilon} \in \{0,1\}^2} (-1)^{\varepsilon_1 + \varepsilon_2} \bigg(\widetilde{\theta}_{\bm{\mu}+\bm{\varepsilon}, \bm{c_1}}^{(2)}(\tau) + \widetilde{\theta}_{\bm{\mu}+\bm{\varepsilon}, \bm{c_2}}^{(2)}(\tau) \bigg) = 2 \left(\sum_{a, b \geq 0} - \sum_{a, b < 0} \right) (-1)^{a+b} q^{Q(a+\mu_1, b+\mu_2)}.
	\end{align*}
	Using the relation $Q(\bm{x}+\bm{\mu}) = Q(\bm{x}) + Q(\bm{\mu}) + B(\bm{x}, \bm{\mu})$ concludes the proof.
\end{proof}

\begin{theorem}\label{Habiro-false}
	Let
	\begin{align*}
		\bm{\mu_1} &= \frac{1}{2} A^{-1} \pmat{2p+1 \\ 5} = \frac{1}{2(6p-1)} \pmat{6p-7 \\ 6p+3},\\
		\bm{\mu_2} &= \frac{1}{2} A^{-1} \pmat{1 \\ 1} = \frac{1}{2(6p-1)} \pmat{1 \\ 2p-1},\\
		\bm{\mu_3} &= \frac{1}{2} A^{-1} \pmat{2p+3 \\ 3} = \frac{1}{2(6p-1)} \pmat{6p+3 \\ 2p-3},\\
		\bm{\mu_4} &= \frac{1}{2} A^{-1} \pmat{2p+1 \\ 1} = \frac{1}{2(6p-1)} \pmat{6p+1 \\ -2p-1},\\
		\bm{\mu_5} &= \frac{1}{2} A^{-1} \pmat{1 \\ 3} = \frac{1}{2(6p-1)} \pmat{-3 \\ 6p+1},
	\end{align*}
	and
	\begin{align*}
		e_1 = (6p+5)^2, \quad e_2 = 1, \quad e_3 = 36p^2 + 84p+1, \quad e_4 = (6p+1)^2, \quad e_5 = 48p+1.
	\end{align*}
	The five families of the Habiro-type series defined in \cref{Habiro-elements} have the following expressions in terms of false theta functions except for the cases of $p=1$ with $k=1, 3$.
	\begin{align*}
		H_p^{(k)}(q) &= \frac{q^{-\frac{e_k}{24(6p-1)}}}{2\eta(\tau)} \sum_{\bm{\varepsilon} \in \{0,1\}^2} (-1)^{\varepsilon_1 + \varepsilon_2} \bigg(\widetilde{\theta}_{\bm{\mu_k} + \bm{\varepsilon}, \bm{c_1}}^{(2)} (\tau) + \widetilde{\theta}_{\bm{\mu_k} + \bm{\varepsilon}, \bm{c_2}}^{(2)} (\tau) \bigg).
	\end{align*}
\end{theorem}

\begin{proof}
	The cases of $k=1,2,3$ immediately follow from \cref{Hikami-Theorem3.5}, \cref{Habiro-2}, \cref{Habiro-3}, and \cref{Hecke-false}. As for the cases of $k=4,5$, \cref{Hecke-false} can not be applied directly because $\bm{\mu_k}$ does not satisfy the required conditions.
	
	First, we consider the case $k=4$. The difference in proofs for this case comes from
	\[
		\sgn(B(\bm{n} + \bm{\mu_4}, \bm{c_2})) = \sgn(n_2 + \mu_{4,2}) = \begin{cases}
			+1 &\text{if } n_2 > 0,\\
			-1 &\text{if } n_2 \leq 0,
		\end{cases}
	\] 
	which is not equal to $\sgn_0(n_2)$. We recall that, by \cref{Habiro-4},
	\[
		H_p^{(4)}(q) = \frac{1}{(q)_\infty} \left(\sum_{a, b \geq 0} - \sum_{a, b < 0} \right) (-1)^{a+b} q^{\left(p+\frac{1}{2}\right)a^2 + 2ab + \frac{3}{2} b^2 + \frac{2p+1}{2} a + \frac{1}{2} b}.
	\]
	To modify the proof of \cref{Hecke-false}, we add the trivial term
	\[
		-\frac{1}{(q)_\infty} \sum_{a \in \Z} (-1)^a q^{\left(p+\frac{1}{2}\right)a^2 + \left(p+\frac{1}{2}\right)a} = 0
	\]
	to the right-hand side. Then, we get
	\[
		H_p^{(4)}(q) = \frac{1}{(q)_\infty} \left(\sum_{\substack{a \geq 0\\ b>0}} - \sum_{\substack{a < 0 \\ b \leq 0}} \right) (-1)^{a+b} q^{\left(p+\frac{1}{2}\right)a^2 + 2ab + \frac{3}{2} b^2 + \frac{2p+1}{2} a + \frac{1}{2} b}.
	\]
	The rest is similar. As for the case of $k=5$, we also obtain the desired formula by adding the trivial term
	\[
		-\frac{1}{(q)_\infty} \sum_{b \in \Z} (-1)^b q^{\frac{3}{2} b^2 + \frac{3}{2}b} = 0
	\]
	to the expression of $H_p^{(5)}(q)$ in \cref{Habiro-5}.
\end{proof}

We note that for the two excluded cases, the equality no longer holds either. However, the difference seems to be just $1/q$.

Finally, we compute the limit values. For simplicity, we put
\begin{align*}
	\widetilde{\Theta}_{M, \mu}^{(1)} (\tau) &= \widetilde{\theta}_{M, \mu}^{(1)}(\tau) - \widetilde{\theta}_{M, \mu+M/2}^{(1)} (\tau),\\
	\Theta_{M, \mu}^{(1)} (\tau) &= \theta_{M, \mu}^{(1)}(\tau) - \theta_{M, \mu+M/2}^{(1)} (\tau)
\end{align*}
and
\begin{align*}
	\widetilde{\Theta}_{\bm{\mu}, \bm{c}}^{(2)}(\tau) = \sum_{\bm{\varepsilon} \in \{0,1\}^2} (-1)^{\varepsilon_1 + \varepsilon_2} \widetilde{\theta}_{\bm{\mu} + \bm{\varepsilon}, \bm{c}}^{(2)}(\tau)
\end{align*}
for a positive even integer $M > 0$. By the false theta decompositions given in \cref{false-decompose1} and \cref{false-decompose2}, we have the following.

\begin{lemma}\label{THETA-DECOMP}
	The notations are the same as in \cref{false-decompose1} and \cref{false-decompose2}. Then we have
	\begin{align*}
		\widetilde{\Theta}_{\bm{\mu}, \bm{c_1}}^{(2)}(\tau) &= \sum_{j=0}^2 \Theta_{12, m_2+4j}^{(1)}(\tau) \widetilde{\Theta}_{12(6p-1), 3m_1-2m_2-4(6p-1)j}^{(1)}(\tau),\\
		\widetilde{\Theta}_{\bm{\mu}, \bm{c_2}}^{(2)}(\tau) &= \sum_{j=0}^{2p} \Theta_{4(2p+1), m_1+4j}^{(1)}(\tau) \widetilde{\Theta}_{4(2p+1)(6p-1), -2m_1+(2p+1)m_2-4p(6p-1)j}^{(1)}(\tau).
	\end{align*}
\end{lemma}

\begin{proof}
	By \cref{false-decompose1},
	\begin{align*}
		\widetilde{\Theta}_{\bm{\mu}, \bm{c_1}}^{(2)}(\tau) &= \sum_{\bm{\varepsilon} \in \{0,1\}^2} (-1)^{\varepsilon_1 + \varepsilon_2}\\
			&\quad \times \sum_{j=0}^2 \theta_{12, m_2+4\varepsilon_1+6\varepsilon_2+4j}^{(1)}(\tau) \widetilde{\theta}_{12(6p-1), 3m_1 - 2m_2 + 2(6p-1)\varepsilon_1 -4(6p-1)j}^{(1)}(\tau).
	\end{align*}
	Here we use the fact that
	\[
		\pmat{1 \\ 0} = \frac{1}{2} A^{-1} \pmat{4p+2 \\ 4}, \quad \pmat{0 \\ 1} = \frac{1}{2} A^{-1} \pmat{4 \\ 6}.
	\]
	The calculation
	\begin{table}[H]
	\centering
	\begin{tabular}{c||c|c} 
		$(\varepsilon_1, \varepsilon_2)$ & $m_2+4\varepsilon_1+6\varepsilon_2+4j$ & $3m_1 - 2m_2 + 2(6p-1)\varepsilon_1 -4(6p-1)j$ \\ \hline \hline
		$(1,0)$ & $m_2+4(j+1)$ & $3m_1-2m_2+6(6p-1) - 4(6p-1)(j+1)$ \\ \hline
		$(0,1)$ & $m_2 + 6 + 4j$ & $3m_1 - 2m_2 -4(6p-1)j$
	\end{tabular}
	\end{table}
	\noindent yields the result. The same calculation works for the second claim.
\end{proof}

The expressions in \cref{Habiro-false} is rewritten as
\[
	H_p^{(k)}(q) = \frac{q^{-\frac{e_k}{24(6p-1)}}}{2\eta(\tau)} \bigg(\widetilde{\Theta}_{\bm{\mu_k}, \bm{c_1}}^{(2)}(\tau) + \widetilde{\Theta}_{\bm{\mu_k}, \bm{c_2}}^{(2)}(\tau) \bigg).
\]
Then the second term
\[
	\frac{q^{-\frac{e_k}{24(6p-1)}}}{2\eta(\tau)} \widetilde{\Theta}_{\bm{\mu_k}, \bm{c_2}}^{(2)}(\tau)
\]
diverges in the vertical limit $\tau \to 1/N$ because of the Dedekind eta function in the denominator. On the other hand, the remaining first term converges in the same limit since the divergence is canceled out by the decay of $\widetilde{\Theta}_{\bm{\mu_k}, \bm{c_1}}^{(2)} (\tau)$. The difference in convergence comes from the difference of $12$ and $4(2p+1)$ in the subscripts of $\Theta^{(1)}$. To confirm it, we recall the modular transformation formula for the Dedekind eta function, 
\begin{align*}
	\eta(\tau+1) &= e^{\frac{\pi i}{12}} \eta(\tau),\\
	\eta\left(-\frac{1}{\tau}\right) &= (-i\tau)^{1/2} \eta(\tau).
\end{align*}
By combining \cref{theta-modular}, for an even $M$, we have
\begin{align}\label{theta/eta}
\begin{split}
	\frac{\theta_{M,\mu}^{(1)}}{\eta} (\tau+1) &= e^{2\pi i \left(\frac{\mu^2}{2M} - \frac{1}{24}\right)} \frac{\theta_{M,\mu}^{(1)}}{\eta} (\tau),\\
	\frac{\theta_{M,\mu}^{(1)}}{\eta} \left(-\frac{1}{\tau}\right) &= \frac{1}{\sqrt{M}} \sum_{\nu=0}^{M-1} e^{2\pi i \frac{\mu \nu}{M}} \frac{\theta_{M,\nu}^{(1)}}{\eta} (\tau).
\end{split}
\end{align}
These transformations imply the following.

\begin{lemma}
	For an even integer $M > 0$ and a positive integer $N > 0$, we have
	\[
		\frac{\theta_{M,\mu}^{(1)}}{\eta} \left(\frac{1}{N} + it\right) = \frac{1}{M} \sum_{\nu'=0}^{M-1} \left(\sum_{\nu=0}^{M-1} e^{2\pi i \frac{(\mu+\nu')\nu}{M}} e^{-2\pi i N \left(\frac{\nu^2}{2M} - \frac{1}{24}\right)} \right) \frac{\theta_{M,\nu'}^{(1)}}{\eta} \left(-\frac{1}{N} + \frac{i}{N^2t} \right).
	\]
\end{lemma}

\begin{proof}
	It immediately follows from the identity
	\[
		\frac{1}{N} + it = -\cfrac{1}{-N-\cfrac{1}{-\cfrac{1}{N} + \cfrac{i}{N^2 t}}}
	\]
	and the transformation formulas in \eqref{theta/eta}.
\end{proof}

This lemma implies that
\begin{align*}
	\frac{\Theta_{M, \mu}^{(1)}}{\eta} \left(\frac{1}{N} + it \right) &= \frac{2}{M} \sum_{\nu'=0}^{M-1} \left( \sum_{\substack{0 \leq \nu \leq M-1 \\ \nu: \text{odd}}} e^{2\pi i \frac{(\mu+\nu') \nu}{M}} e^{-2\pi i N \left(\frac{\nu^2}{2M} - \frac{1}{24}\right)} \right) \frac{\theta_{M,\nu'}^{(1)}}{\eta} \left(-\frac{1}{N} + \frac{i}{N^2t} \right)\\
		&= \frac{2}{M} \sum_{\nu'=0}^{M-1} \left(e^{2\pi i \left(\frac{\mu+\nu'}{M} - \frac{N}{2M} + \frac{N}{24}\right)} \sum_{\nu=0}^{M/2-1} e^{2\pi i \frac{-N \nu^2 + (\mu+\nu'-N)\nu}{M/2}} \right) \frac{\theta_{M, \nu'}^{(1)}}{\eta} \left(-\frac{1}{N} + \frac{i}{N^2t} \right).
\end{align*}
Furthermore, we take $M = 4(2p+1)$ and put $b = \mu + \nu'$. Then the inner sum becomes
\begin{align*}
	\sum_{\nu=0}^{M/2-1} e^{2\pi i \frac{-N \nu^2 + (\mu+\nu'-N)\nu}{M/2}} &= \sum_{\nu=0}^{4p+1} e^{2\pi i \frac{-N (\nu+2p+1)^2 + (b-N)(\nu + 2p+1)}{2(2p+1)}} = (-1)^b \sum_{\nu=0}^{4p+1} e^{2\pi i \frac{-N \nu^2 + (b-N)\nu}{2(2p+1)}}.
\end{align*}
The equation shows that if $b = \mu + \nu'$ is odd, the inner sum equals $0$.

\begin{lemma}\label{theta-eta-limit}
	For a positive integer $p \geq 1$ and an odd $\mu$, we have
	\begin{align*}
		\frac{\Theta_{4(2p+1), \mu}^{(1)}}{\eta} \left(\frac{1}{N} + it \right) &= \frac{1}{2(2p+1)} \sum_{\nu'=0}^{2(2p+1)-1} e^{2\pi i \left(\frac{2\nu'+\mu+1}{4(2p+1)} - \frac{N}{8(2p+1)} + \frac{N}{24}\right)} \\
			&\qquad \times \sum_{\nu=0}^{2(2p+1)-1} e^{2\pi i \frac{-N \nu^2 + (2\nu'+\mu+1-N)\nu}{2(2p+1)}} \frac{\theta_{4(2p+1), 2\nu'+1}^{(1)}}{\eta} \left(-\frac{1}{N} + \frac{i}{N^2t} \right).
	\end{align*}
\end{lemma}

In the particular case of $p=1$, we obtain the following converging limit formula.

\begin{corollary}\label{THETA-eta-lim}
	For an odd $\mu$, we have
	\[
		\lim_{t \to 0} \frac{\Theta_{12,\mu}^{(1)}}{\eta} \left(\frac{1}{N} + it \right) = \begin{cases}
			1 &\text{if } \mu \equiv 1, 11 \pmod{12},\\
			-1 &\text{if } \mu \equiv 5, 7 \pmod{12},\\
			0 &\text{if otherwise}.
		\end{cases}
	\]
\end{corollary}

\begin{proof}
	Since $\mu$ is odd, by applying \cref{theta-eta-limit}, we obtain
	\begin{align*}
		\frac{\Theta_{12, \mu}^{(1)}}{\eta} \left(\frac{1}{N} + it \right) &= \frac{1}{6} \sum_{\nu'=0}^{5} e^{2\pi i \left(\frac{2\nu'+\mu+1}{12} \right)} \\
			&\qquad \times \sum_{\nu=0}^{5} e^{2\pi i \frac{-N \nu^2 + (2\nu'+\mu+1-N)\nu}{6}} \frac{\theta_{12, 2\nu'+1}^{(1)}}{\eta} \left(-\frac{1}{N} + \frac{i}{N^2t} \right).
	\end{align*}
	By the definition of the theta function $\theta_{M,\mu}^{(1)}(\tau)$ in \cref{theta-modular}, we see that
	\[
		\lim_{\tau \to i\infty} q^{-1/24} \theta_{12, 2\nu'+1}^{(1)}(\tau) = \begin{cases}
			1 &\text{if } 2\nu'+1 \equiv 1, 11 \pmod{12},\\
			0 &\text{if otherwise}.
		\end{cases}
	\]
	Therefore,
	\begin{align*}
		\lim_{t \to 0} \frac{\Theta_{12, \mu}^{(1)}}{\eta} \left(\frac{1}{N} + it \right) &= \frac{1}{6} e^{2\pi i \frac{\mu+1}{12}} \sum_{\nu=0}^{5} e^{2\pi i \frac{-N \nu^2 + (\mu+1-N)\nu}{6}} + \frac{1}{6} e^{2\pi i \frac{\mu-1}{12}} \sum_{\nu=0}^{5} e^{2\pi i \frac{-N \nu^2 + (\mu-1-N)\nu}{6}}\\
			&= \begin{cases}
				1 &\text{if } \mu \equiv 1, 11 \pmod{12},\\
				-1 &\text{if } \mu \equiv 5, 7 \pmod{12},\\
				0 &\text{if otherwise},
			\end{cases}
	\end{align*}
	which concludes the proof.
\end{proof}

\begin{corollary}\label{Theta2-conv}
	For any pair of odd integers $(m_1, m_2) \in \Z^2$, we put $\bm{\mu} = \frac{1}{2} A^{-1} \smat{m_1 \\ m_2}$. Then we have
	\begin{align*}
		&\lim_{\tau \to 1/N} \frac{\widetilde{\Theta}_{\bm{\mu}, \bm{c_1}}^{(2)} (\tau)}{\eta (\tau)}\\
		&= \widetilde{\theta}_{12(6p-1), 3m_1-2m_2-4(6p-1)j_1}^{(1)} (1/N) - \widetilde{\theta}_{12(6p-1), 3m_1-2m_2-4(6p-1)j_1+6(6p-1)}^{(1)} (1/N)\\
		&\qquad - \widetilde{\theta}_{12(6p-1), 3m_1-2m_2-4(6p-1)j_2}^{(1)} (1/N) + \widetilde{\theta}_{12(6p-1), 3m_1-2m_2-4(6p-1)j_2+6(6p-1)}^{(1)} (1/N),
	\end{align*}
	where we put
	\begin{align*}
		(j_1, j_2) = \begin{cases}
			(0, 1) &\text{if } m_2 \equiv 1 \pmod{12},\\
			(2,1) &\text{if } m_2 \equiv 3 \pmod{12},\\
			(2,0) &\text{if } m_2 \equiv 5 \pmod{12},\\
			(1,0) &\text{if } m_2 \equiv 7 \pmod{12},\\
			(1,2) &\text{if } m_2 \equiv 9 \pmod{12},\\
			(0,2) &\text{if } m_2 \equiv 11 \pmod{12}.
		\end{cases}
	\end{align*}
	Here we put $\widetilde{\theta}_{M, \mu}^{(1)} (1/N) = \lim_{\tau \to 1/N} \widetilde{\theta}_{M, \mu}^{(1)} (\tau).$
\end{corollary}

\begin{proof}
	It immediately follows from \cref{THETA-DECOMP} and \cref{THETA-eta-lim}.
\end{proof}

\begin{theorem}\label{main}
	The limit of the first half of the expression of each of five Habiro-type series given in \cref{Habiro-false} converges as $\tau \to 1/N$. More precisely, we have
	\begin{align*}
		\lim_{\tau \to 1/N} \frac{q^{-\frac{e_k}{24(6p-1)}}}{2 \eta(\tau)} \sum_{\bm{\varepsilon} \in \{0,1\}^2} (-1)^{\varepsilon_1 + \varepsilon_2} \widetilde{\theta}_{\bm{\mu_k}+\bm{\varepsilon}, \bm{c_1}}^{(2)} (\tau) 
			&= -\frac{1}{2} \lim_{\tau \to 1/N} q^{-\frac{e_k}{24(6p-1)}} \widetilde{\Phi}_{(2,3,6p-1)}^{(1,1,\ell_k)} (\tau),
	\end{align*}
	where $(\ell_k)_{1 \leq k \leq 5} = (1, p, 2p-1, 2p, 3p-1)$.
\end{theorem}

\begin{proof}
	By \cref{Theta2-conv}, the limit on the left-hand side converges. The subscripts of the resulting false theta functions are listed below.
	\begin{table}[H]
	\centering
	\begin{tabular}{c||c|c} 
		$(m_1, m_2)$ & $3m_1 - 2m_2 - 4(6p-1)j_1$ & $3m_1-2m_2-4(6p-1)j_2$ \\ \hline \hline
		$(2p+1,5)$ & $-42p+1$ & $6p-7$ \\ \hline
		$(1,1)$ & $1$ & $-24p+5$ \\ \hline
		$(2p+3,3)$ & $-42p+11$ & $-18p+7$ \\ \hline
		$(2p+1,1)$ & $6p+1$ & $-18p+5$ \\ \hline
		$(1,3)$ & $-48p+5$ & $-24p+1$
	\end{tabular}
	\end{table}
	For $\bm{p} = (2,3,6p-1)$, the values of 
	\[
		\mu(\bm{\varepsilon}, \bm{\ell}) = 6(6p-1) \left(1 + \frac{\varepsilon_1 \ell_1}{2} + \frac{\varepsilon_2 \ell_2}{3}+ \frac{\varepsilon_3 \ell_3}{6p-1} \right) \mod{12(6p-1)}
	\]
	defined in \cref{Eichler-false} coincide with the values in the above list for some $(\bm{\varepsilon}, \bm{\ell})$'s as follows.
	\begin{table}[H]
	\centering
	\begin{tabular}{l||c|c||c|c} 
		$\bm{\ell} = (\ell_1, \ell_2, \ell_3)$ & $\bm{\varepsilon}$ & $\mu(\bm{\varepsilon}, \bm{\ell})$ & $\bm{\varepsilon}$ & $\mu(\bm{\varepsilon}, \bm{\ell})$ \\ \hline \hline
		$\bm{\ell_1} = (1,1,1)$ & $(-1,1,-1)$ & $-42p+1$ & $(-1,-1,-1)$ & $6p-7$ \\ \hline
		$\bm{\ell_2} = (1,1,p)$ & $(1,1,1)$ & $1$ & $(1,-1,1)$ & $-24p+5$ \\ \hline
		$\bm{\ell_3} = (1,1,2p-1)$ & $(1,-1,-1)$ & $-42p+11$ & $(1,1,-1)$ & $-18p+7$ \\ \hline
		$\bm{\ell_4} = (1,1,2p)$ & $(1,1,1)$ & $6p+1$ & $(1,-1,1)$ & $-18p+5$ \\ \hline
		$\bm{\ell_5} = (1,1,3p-1)$ & $(-1,-1,1)$ & $-48p+5$ & $(-1,1,1)$ & $-24p+1$
	\end{tabular}
	\end{table}
	By the relations $\mu(-\bm{\varepsilon}, \bm{\ell}) \equiv -\mu(\bm{\varepsilon}, \bm{\ell})$ and $\mu((-\varepsilon_1, \varepsilon_2, \varepsilon_3), \bm{\ell}) \equiv \mu((\varepsilon_1, \varepsilon_2, \varepsilon_3), \bm{\ell}) + 6(6p-1)$ for $\ell_1 = 1$, we have
	\begin{align*}
		\widetilde{\Phi}_{\bm{p}}^{\bm{\ell_1}}(\tau) &= \widetilde{\theta}_{12(6p-1), 6p-7}^{(1)} (\tau) - \widetilde{\theta}_{12(6p-1), -42p+1}^{(1)} (\tau)\\
			&\qquad - \widetilde{\theta}_{12(6p-1), 6p-7+6(6p-1)}^{(1)} (\tau) + \widetilde{\theta}_{12(6p-1), -42p+1+6(6p-1)}^{(1)} (\tau),
	\end{align*}
	which implies that
	\[
		\lim_{\tau \to 1/N} \frac{\widetilde{\Theta}_{\bm{\mu_1}, \bm{c_1}}^{(2)}(\tau)}{\eta(\tau)} = - \lim_{\tau \to 1/N} \widetilde{\Phi}_{\bm{p}}^{\bm{\ell_1}}(\tau).
	\]
	The same calculation works for the remaining four cases.
\end{proof}

As for the second half of the expressions of the Habiro-type series, the corresponding limit
\[
	\lim_{t \to 0} \frac{\Theta_{4(2p+1), m_1+4j}^{(1)}}{\eta} \left(\frac{1}{N} + it\right)
\]
to \cref{THETA-eta-lim} diverges in general because of the Dedekind eta function in the denominator. In other words, the limit $\lim_{q \to e^{2\pi i/N}} H_p^{(k)}(q)$ from within the unit disc diverges in general.

\subsection{Hikami's question, revisited}

In \cite[Concluding remarks]{Hikami2007}, Hikami left the question on the modular transformation theory of the Hecke-type series expression of the Habiro-type series as a future study. Hikami's question locates in the counterpart of the transformation theory of the indefinite-theta function expressions of Ramanujan's mock theta functions developed by Zwegers~\cite{Zwegers2002}. As mentioned at the beginning of \cref{s3}, the work of Bringmann--Nazaroglu~\cite{BringmannNazaroglu2019} on the transformation theory of false theta functions is one answer to this question. In this last subsection, we will briefly review it.

We now consider the false theta function 
\[
	\widetilde{\theta}_{\bm{\mu}, \bm{c}}^{(2)}(\tau) = \sum_{\bm{n} \in L + \bm{\mu}} \sgn(B(\bm{n}, \bm{c})) q^{Q(\bm{n})}
\]
for the general setting in \cref{def-false-theta} with $\mathrm{rank} L = 2$. Recalling the definition of $Q_{\bm{c}}(\bm{x})$ defined in \cref{heihoukansei} and the expression
\[
	\widetilde{\theta}_{\bm{\mu}, \bm{c}}^{(2)}(\tau) = \sum_{\bm{n} \in L + \bm{\mu}} \sgn(B(\bm{n}, \bm{c})) q^{\frac{1}{2} B(\bm{n}, \bm{c})^2} q^{Q_{\bm{c}}(\bm{n})},
\] 
we introduce the function
\[
	f_{\tau, z}(\bm{x}) = B(\bm{x}, \bm{c}) e^{\pi i z B(\bm{x}, \bm{c})^2} e^{2\pi i \tau Q_{\bm{c}}(\bm{x})} = B(\bm{x}, \bm{c}) e^{2\pi i Q(\bm{x}) \tau} e^{\pi i (z-\tau) B(\bm{x}, \bm{c})^2}
\]
for $\tau, z \in \bbH$ and $\bm{x} \in \R^2$.

\begin{lemma}
	We have
	\[
		\mathcal{F}(f_{\tau, z}) (\bm{x}) = \frac{-i (-i\tau)^{-1/2} (-iz)^{-3/2}}{\sqrt{\det A}} f_{-\frac{1}{\tau}, -\frac{1}{z}} (\bm{x}),
	\]
	where
	\[
		\mathcal{F}(f) (\bm{x}) = \int_{\R^2} f(\bm{y}) e^{-2\pi i B(\bm{x}, \bm{y})} d\bm{y}
	\]
	is the Fourier transform of $f$.
\end{lemma}

\begin{proof}
	The idea of the proof is based on Bringmann--Nazaroglu~\cite{BringmannNazaroglu2019}. By the definition,
	\begin{align*}
		\mathcal{F}(f_{\tau, z}) (\bm{x}) = \int_{\R^2} B(\bm{y}, \bm{c}) e^{2\pi i Q(\bm{y})\tau} e^{\pi i (z- \tau) B(\bm{y}, \bm{c})^2} e^{-2\pi i B(\bm{x}, \bm{y})} d\bm{y}.
	\end{align*}
	Let $A = \smat{a & b \\ b & d}$ and $\bm{c} = \smat{c_1 \\ c_2}$, we put
	\[
		C= \frac{1}{\sqrt{2Q(\bm{c})}} \pmat{a c_1 + b c_2 & c_2 \sqrt{\det A} \\ b c_1 + d c_2 & -c_1 \sqrt{\det A}}.
	\]
	Then we have $A = C C^T$ and $C^T \bm{c} = \smat{\sqrt{2Q(\bm{c})} \\ 0}$. Since $2Q(\bm{c}) = 1$, we have $C^T \bm{c} = \smat{1 \\ 0}$. By putting $\smat{y_1 \\ y_2} = C^T \bm{y}$ and $\smat{x_1 \\ x_2} = C^T \bm{x}$,
	\begin{align*}
		\mathcal{F}(f_{\tau, z}) (\bm{x}) &= \int_{\R^2} y_1 e^{\pi i (y_1^2 + y_2^2) \tau} e^{\pi i (z- \tau) y_1^2} e^{-2\pi i (x_1 y_1 + x_2 y_2)} \frac{dy_1 dy_2}{\sqrt{\det A}}\\
			&= \frac{1}{\sqrt{\det A}}\int_{-\infty}^\infty y_1 e^{\pi i y_1^2 z - 2\pi i x_1 y_1} dy_1 \int_{-\infty}^\infty e^{\pi i y_2^2 \tau - 2\pi i x_2 y_2} dy_2.
	\end{align*}
	Each integral is well-known and is equal to
	\begin{align*}
			&= \frac{-i}{\sqrt{\det A}} (-i \tau)^{-1/2} (-i z)^{-3/2} x_1 e^{\pi i x_1^2 \frac{-1}{z}}  e^{\pi i x_2^2 \frac{-1}{\tau}}\\
			&= \frac{-i}{\sqrt{\det A}} (-i \tau)^{-1/2} (-i z)^{-3/2} B(\bm{x}, \bm{c}) e^{2\pi i Q(\bm{x}) \frac{-1}{\tau}} e^{\pi i (\frac{-1}{z} - \frac{-1}{\tau}) B(\bm{x}, \bm{c})^2},
	\end{align*}
	which equals the right-hand side of the desired equation.
\end{proof}

\begin{lemma}\label{modular-g}
	We define the bivariate theta function $g_{\bm{\mu}, \bm{c}} (\tau,z)$ by
	\[
		g_{\bm{\mu}, \bm{c}} (\tau,z) = \sum_{\bm{n} \in L + \bm{\mu}} f_{\tau, z}(\bm{n}) = \sum_{\bm{n} \in L + \bm{\mu}} B(\bm{n}, \bm{c}) e^{\pi i z B(\bm{n}, \bm{c})^2} q^{Q_{\bm{c}}(\bm{n})}.
	\]
	Then we have
	\[
		g_{\bm{\mu}, \bm{c}} \left(-\frac{1}{\tau}, -\frac{1}{z}\right) = \frac{-i (-i\tau)^{1/2} (-iz)^{3/2}}{\vol(\R^2/L) \sqrt{\det A}} \sum_{\bm{\nu} \in L^*/L} e^{2\pi i B(\bm{\nu}, \bm{\mu})} g_{\bm{\nu}, \bm{c}} (\tau, z).
	\]
\end{lemma}

\begin{proof}
	It follows from Poisson's summation formula
	\[
		\vol(\R^2/L) \sum_{\bm{n} \in L} f(\bm{n}+\bm{x}) = \sum_{\bm{n} \in L^*} \mathcal{F}(f)(\bm{n}) e^{2\pi i B(\bm{n}, \bm{x})}.
	\]
\end{proof}

\begin{lemma}
	\[
		\widetilde{\theta}_{\bm{\mu}, \bm{c}}^{(2)} (\tau) = -i \int_\tau^{i\infty} \frac{g_{\bm{\mu}, \bm{c}}(\tau, z)}{\sqrt{-i(z-\tau)}} dz.
	\]
\end{lemma}

\begin{proof}
	Since
	\[
		\int_\tau^{i \infty} \frac{e^{\pi i a^2 z} dz}{\sqrt{-i(z-\tau)}} = ie^{\pi i a^2 \tau} \int_0^\infty \frac{e^{-\pi a^2t} dt}{\sqrt{t}} = \frac{i e^{\pi i a^2 \tau}}{|a|},
	\]
	we have
	\begin{align*}
		\widetilde{\theta}_{\bm{\mu}, \bm{c}}^{(2)} (\tau) &= -i \sum_{\substack{\bm{n} \in L + \bm{\mu} \\ B(\bm{n}, \bm{c}) \neq 0}} B(\bm{n}, \bm{c}) \frac{i e^{\pi i B(\bm{n}, \bm{c})^2 \tau}}{|B(\bm{n}, \bm{c})|} q^{Q_{\bm{c}}(\bm{n})}\\
			&= -i \sum_{\bm{n} \in L + \bm{\mu}} B(\bm{n}, \bm{c}) \int_\tau^{i\infty} \frac{e^{\pi i z B(\bm{n}, \bm{c})^2} dz}{\sqrt{-i(z-\tau)}} q^{Q_{\bm{c}}(\bm{n})},
	\end{align*}
	which finishes the proof.
\end{proof}

The above integral expression and the modular transformation of $g_{\bm{\mu}, \bm{c}}(\tau,z)$ yield the following $S$-transformation formula.

\begin{theorem}\label{S-trans-ft2}
	We assume that $\Re(\tau) \neq 0$. Then we have
	\begin{align*}
		(-i \tau)^{-1} \widetilde{\theta}_{\bm{\mu}, \bm{c}}^{(2)} \left(-\frac{1}{\tau}\right) = -\frac{i \sgn(\Re(\tau))}{\vol(\R^2/L) \sqrt{\det A}} \sum_{\bm{\nu} \in L^*/L} e^{2\pi i B(\bm{\mu}, \bm{\nu})} \int_0^\tau \frac{g_{\bm{\nu}, \bm{c}} (\tau, z)}{\sqrt{-i(z-\tau)}} dz.
	\end{align*}
\end{theorem}

\begin{proof}
	By changing a variable via $z = -1/z'$, we have
	\begin{align*}
		(-i \tau)^{-1} \widetilde{\theta}_{\bm{\mu}, \bm{c}}^{(2)} \left(-\frac{1}{\tau}\right) = -i(-i\tau)^{-1} \int_0^\tau \frac{g_{\bm{\mu}, \bm{c}} \left(-\frac{1}{\tau}, -\frac{1}{z'} \right) }{\sqrt{-i \left(-\frac{1}{z'}-\frac{-1}{\tau}\right)}} \frac{dz'}{(-iz')^2}.
	\end{align*}
	By \cref{modular-g}, it becomes
	\[
		= -\frac{1}{\vol(\R^2/L) \sqrt{\det A}} \sum_{\bm{\nu} \in L^*/L} e^{2\pi iB(\bm{\nu}, \bm{\mu})} \int_0^\tau \frac{g_{\bm{\nu}, \bm{c}}(\tau, z)}{\sqrt{i(z-\tau)}} dz.
	\]
	Since $\sqrt{i(z-\tau)} = -i \sgn(\Re(\tau)) \sqrt{-i(z-\tau)}$ holds for $z \in \bbH$ on the line segment connecting $0$ and $\tau$, we have the desired result.
\end{proof}

By adding
\begin{align*}
	&\frac{\sgn(\Re(\tau))}{\vol(\R^2/L) \sqrt{\det A}} \sum_{\bm{\nu} \in L^*/L} e^{2\pi i B(\bm{\mu}, \bm{\nu})} \widetilde{\theta}_{\bm{\nu}, \bm{c}}^{(2)}(\tau)\\
	&= -\frac{i \sgn(\Re(\tau))}{\vol(\R^2/L) \sqrt{\det A}} \sum_{\bm{\nu} \in L^*/L} e^{2\pi i B(\bm{\mu}, \bm{\nu})} \int_\tau^{i\infty} \frac{g_{\bm{\nu}, \bm{c}} (\tau, z)}{\sqrt{-i(z-\tau)}} dz
\end{align*}
to the both sides of \cref{S-trans-ft2}, we also obtain the following. 

\begin{corollary}
	We assume that $\Re(\tau) \neq 0$. Then we have
	\begin{align*}
		&(-i \tau)^{-1} \widetilde{\theta}_{\bm{\mu}, \bm{c}}^{(2)} \left(-\frac{1}{\tau}\right) + \frac{\sgn(\Re(\tau))}{\vol(\R^2/L) \sqrt{\det A}} \sum_{\bm{\nu} \in L^*/L} e^{2\pi i B(\bm{\mu}, \bm{\nu})} \widetilde{\theta}_{\bm{\nu}, \bm{c}}^{(2)} (\tau)\\
			&= -\frac{i \sgn(\Re(\tau))}{\vol(\R^2/L) \sqrt{\det A}} \sum_{\bm{\nu} \in L^*/L} e^{2\pi i B(\bm{\mu}, \bm{\nu})} \int_0^{i\infty} \frac{g_{\bm{\nu}, \bm{c}} (\tau, z)}{\sqrt{-i(z-\tau)}} dz,
	\end{align*}
	where the integration path avoids the branch cut defined by $\sqrt{-i(z-\tau)}$, that is, $\{z = \tau - iu \in \C \mid u>0\}$.
\end{corollary}

If a shape similar to \cref{S-trans-ft1} is desired, we can re-apply \cref{modular-g} to the right-hand side.

\section*{Acknowledgements}
The author would like to express his sincere gratitude to Kazuhiro Hikami for his introduction to the theory of quantum invariants in a series of lectures and seminars. The author is also grateful to Yuya Murakami, Shin-ichiro Seki, and Shoma Sugimoto for continuous helpful communication. The work was supported by JSPS KAKENHI Grant Number JP20K14292 and JP21K18141.

\bibliographystyle{amsplain}
\bibliography{References}

\end{document}